\documentclass[journal, twocolumn, 10pt]{IEEEtran}
\usepackage{amsfonts,amsmath,amssymb,amsthm}
\usepackage[linesnumbered, ruled]{algorithm2e}
\usepackage{times}
\usepackage{bm}
\usepackage{color}
\usepackage{balance}
\usepackage{graphicx}
\usepackage{setspace}
\usepackage{comment}
\usepackage{cases}
\graphicspath{{./eps/}}
\usepackage{here}
\usepackage{ascmac} 
\usepackage{array}
\usepackage{float}
\usepackage{cite}
\usepackage{url}
\usepackage{subfigure}
\newtheorem{mydef}{Definition}
\newtheorem{mylem}{Lemma}
\newtheorem{mypro}{Problem}

\newtheorem{mythm}{Theorem}
\newtheorem{myas}{Assumption}
\newtheorem{myrem}{Remark}

\newcommand{\rfig}[1]{Fig.\,\ref{#1}} 
 
\newcommand{\req}[1]{\eqref{#1}} 
 
\newcommand{\rtab}[1]{Table\,\ref{#1}}
\newcommand{\rlem}[1]{Lemma\,\ref{#1}}

\newcommand{\rrem}[1]{Remark\,\ref{#1}}
\newcommand{\rsec}[1]{Section\,\ref{#1}}
\newcommand{\rdef}[1]{Definition\,\ref{#1}}
\newcommand{\ras}[1]{Assumption\,\ref{#1}}
\newcommand{\ralg}[1]{Algorithm\,\ref{#1}}
\newcommand{\rpro}[1]{Problem\,\ref{#1}}
\newcommand{\rline}[1]{line\,\ref{#1}}

\newcommand{\rthm}[1]{Theorem\,\ref{#1}}

\SetKwInOut{Parameter}{parameter}
\newcommand{\qedwhite}{\hfill \ensuremath{\Box}}

\begin{document}
\title{Synthesizing communication plans for reachability and safety specifications}
\author{\large Kazumune~Hashimoto, Dimos~V.~Dimarogonas,~\IEEEmembership{Senior Member,~IEEE}
\thanks{Kazumune Hashimoto is with the School of Electrical Engineering and Computer Science, KTH Royal Institute of Technology, 10044 Stockholm, Sweden (e-mail: kazumune.hashimoto@z5.keio.jp). His work is supported by the Knut and Alice Wallenberg Foundation.}
\thanks{Dimos V. Dimarogonas is with the ACCESS Linnaeus Center,
School of Electrical Engineering and Computer Science, KTH Royal Institute of Technology, 10044 Stockholm, Sweden (e-mail : dimos@ee.kth.se). His work was supported by the Swedish Research council (VR) and the Knut and Alice Wallenberg Foundation.}
}
\date{}
\maketitle

\begin{abstract}
We propose control and communication strategies for nonlinear networked control systems subject to state and input constraints. The objective is to steer the state of the system towards a prescribed target set in finite time (\textit{reachability}), while at the same time remaining inside a safety set for all time (\textit{safety}). 
By leveraging the notion of $\delta$-ISS control Lyapunov function, we derive a sufficient condition to generate a communication scheduling, such that the resulting state trajectory guarantees reachability and safety. Moreover, in order to alleviate computational burden we present a way to find a suitable communication scheduling by implementing abstraction schemes and standard graph search methodologies. Simulation examples validate the effectiveness of the proposed approach. 
\end{abstract}


\section{Introduction}
\IEEEPARstart{W}{ith}  the advent of communication technologies, 
there has been a growing trend of introducing communication networks in many control applications, such as manufacturing plants, autonomous robots, traffic systems, and so on \cite{gupta2010}. Typically, a control system whose sensors, actuators, and controllers are spatially distributed and connected over communication channels is referred to as a \textit{Networked Control System} (NCS). 
On one hand, the introduction of NCSs has many advantages, such as the elimination of redundant wirings, the availability to control a plant {remotely} in distant areas, and so on \cite{wei2001}. On the other hand, the introduction of NCSs has raised new technological challenges that remain to be solved. 
In particular, one of the crucial challenges lies in the fact that NCSs are subject to \textit{limited resources}, such as limited life-time of battery powered devices and a limited communication bandwidth. For example, sensors and relay nodes are typically battery driven and are equipped with a frugal battery capacity. Thus, designing appropriate feedback controllers to save energy consumption is a crucial problem to be solved. 
In order to reduce redundant utilizations of such limited resources, two relevant control schemes have been proposed, namely,
\textit{event-triggered control} and \textit{self-triggered control} \cite{heemels2012a}. 
In both strategies, the objective is to reduce the communication frequency between the plant and the controller. Specifically, sensor data and control signals are exchanged over a communication network \textit{only when} they are needed, so that communication is given aperiodically.  
Such aperiodic scheme can potentially lead to energy savings of battery powered devices, since the communication over the network is known to be one of the crucial energy consumers. 

So far, event and self-triggered control strategies have been analyzed for many different types of systems, including linear systems \cite{tabuada2007a,heemels2011a,dolk2017c,hashimoto2018b,kishida2018}, nonlinear systems \cite{romain2014a,tabuada2010a}, and distributed control systems \cite{dimos2012a}. In addition, more sophisticated approaches to reduce sensing and communication costs have been provided, such as periodic event-triggered control \cite{heemels2013a,potoyan2013} and dynamic event-triggered control\cite{girard2014a}. Some experimental validations of applying the event-triggered and self-trigered control schemes have also been provided, see e.g., in \cite{araujo2013,chen2016a}. For more different formulations and approaches, see \cite{eventsurvey} for a recent survey paper. 

 
In this paper, we consider the following problem: {``design control and communication strategies, such that the state of the system is steered towards a given target region (\textit{reachability}), while at the same time remaining inside a given safety set for all times (\textit{safety})''}. In other words, we present aperiodic control strategies for achieving \textit{reachability} and \textit{safety}, in contrast to the afore-cited approaches in which the control objective is mostly stabilization (of the origin) or output regulation.  
Reachability and safety controller synthesis have been active areas of research in various control applications, such as flight control systems\cite{aircraftsafety}, motion planning of dynamic robots\cite{gerogios2009a}, safe platooning or control of maneuvers \cite{saferobotcontrol2016a}, to name a few, and many different theoretical foundations and problem formulations have been already proposed, see e.g., \cite{habets2004a,belta2004a,temporalmpc2012a,girard2012a,girard2009a}. 
For example, in \cite{habets2004a}, reachability analysis is given for continuous-time linear systems on a set of full-dimensional polytopes (or simplices) that are partitioned in a state-space. A piece-wise affine control law is designed as a set of vector fields to steer the state to exit a prescribed facet in finite time to enter an adjacent polytope.  Another approach to controller synthesis problem is based on {approximately bisimilar abstractions}\cite{girard2012a,girard2009a}. In this approach, a symbolic model that approximately simulates the behavior of the original control system is constructed through the notion of approximate bisimilar relations, and a safety controller is synthesized based on finding appropriate paths by solving symbolic optimal control problems. 

While many control strategies to achieve reachability and safety have been proposed as illustrated above, only a few works have been provided to accommodate \textit{communication strategies}, aiming at reducing the communication load for NCSs, see e.g., \cite{hashimoto2018b,hashimoto2017d}. 
For example, in \cite{hashimoto2018b} an aperiodic control scheme was proposed by using the notion of control invariant set, which guarantees the existence of a controller such that the state remains inside the safety set for all time. However, a fundamental assumption required in that paper is that the safety set is \textit{convex}; essentially, this assumption is required to obtain the invariant set 
as well as to make the optimal control problem convex. 
Thus, designing suitable control and communication strategies for a \textit{non-convex} safety set may still be a challenging problem, which is the case considered in this paper. 


The approach presented here differs from the existing event and self-triggered control strategies in the following way. We start by defining the notion of \textit{$\delta$-ISS control Lyapunov function}, which is a variant of $\delta$-ISS Lyapunov functions\cite{incremental,incremental_discrete} that are defined for systems without controls. This function is a useful tool to analyze contractive behaviors between any pair of the state trajectories, and has been attracted much attention in various analysis and control design applications, see, e.g., \cite{girard2009a,girard2010a}. 
Based on this function, we next introduce the notion of an \textit{error propagation model}, which quantifies how a closed loop state trajectory can track a given reference according to the occurrence or non-occurrence of communication. The error propagation model is a key ingredient to derive a sufficient condition to generate a communication scheduling that guarantees reachability and safety. Moreover, in order to alleviate the computational burden to find a suitable communication scheduling, 
we next translate the error propagation model into a \textit{symbolic error system}, which is a variant of transition systems. 
By this translation, a suitable communication scheduling to achieve reachability and safety can be efficiently found by implementing standard graph search algorithms.

The proposed approach builds upon our previous work \cite{hashimoto2017d}. In \cite{hashimoto2017d}, we construct a collection of polyhedral contractive sets with different inter-event times and control updates. Then, these sets are translated into the corresponding symbolic system, and the communication scheduling is generated by graph search algorithms. However, the previous approach is only applicable to \textit{linear} discrete-time systems with a \textit{convex} safety set, and, moreover, the control objective is the stabilization to the origin. In the approach presented in this paper, we deal with \textit{nonlinear} discrete-time systems and a \textit{non-convex} safety set and the control objective is to achieve reachability and safety. Thus the proposed approach establishes the novelty with respect to the previous approach in \cite{hashimoto2017d}. 
{As we will see later, the key ideas to achieve this novelty is the utilization of several trajectory generation tools, such as RRT\cite{lavalle1999}, RRT*\cite{karaman2010}, which allows to generate a nominal state trajectory towards a target region in a non-convex safety set (see \rsec{offline_control_sec}), and the utilization of $\delta$-ISS control Lyapunov functions, which allows to analyze how the actual state trajectory differs from the nominal one according to the occurrence of communication (see \rsec{error_propagation_sec}). }


The rest of the paper is organized as follows. We provide some premilinaries and the problem formulation in Section~II. 
The control and communication strategies are proposed in Section~III and IV, respectively. In Section~V, simulation results are given to validate the effectiveness of the proposal including linear and nonlinear systems. Finally, conclusions and future works are provided in Section~VI. \\




\noindent
\textbf{Notations.} Let $\mathbb{R}$, $\mathbb{R}_+$, $\mathbb{N}$, $\mathbb{N}_+$ be the non-negative real, {positive real, non-negative integers}, and {positive integers}, respectively. We denote $\mathbb{N}_{a: b}$ as the set of integers in the interval $[a, b]$. 
A function $\alpha : \mathbb{R}\rightarrow \mathbb{R}$ is called a class ${\cal K}$-function if it is continuous, strictly increasing, and $\alpha (0) = 0$. It is called a class ${\cal K}_{\infty}$-function if it is a class ${\cal K}$-function and $\alpha (r) \rightarrow \infty$ as $r \rightarrow \infty$. We denote by ${\rm Id}:\mathbb{R}\rightarrow \mathbb{R}$ the identity function, i.e., ${\rm Id}(r) = r$, $\forall r\geq 0$.
{The notation $\alpha_1 \circ \alpha_2$ is used to denote the composition of the two functions $\alpha_1$ and $\alpha_2$.} The notation $\sigma_{\max} (A)$ is used to denote the maximum singular value of the matrix $A$. We denote by $\|x\|$ the Euclidean norm of vector $x$. 

\begin{figure}[t]
  \begin{center}
   \includegraphics[width=7cm]{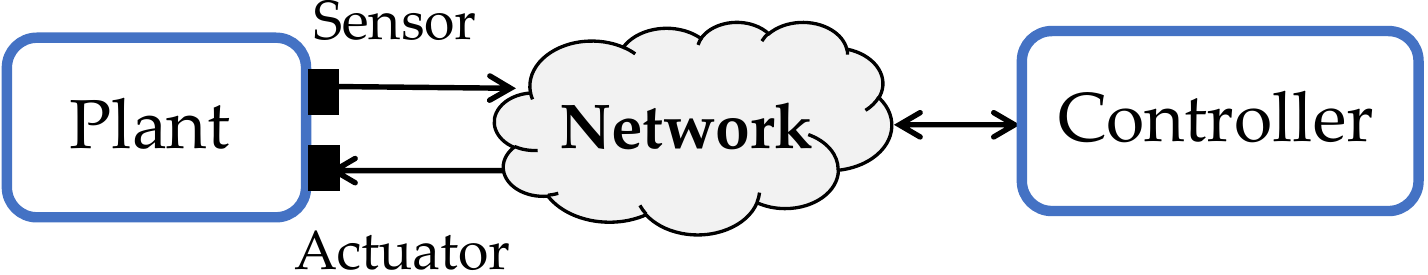}
   \caption{Networked Control Systems} 
   \label{NCS}
  \end{center}
\end{figure}
\section{Problem formulation}
\subsection{Plant dynamics, free-space}
Let us consider a networked control system depicted in \rfig{NCS}, where the plant and the controller are connected over a communication network. We assume that the dynamics of the plant is given by the following nonlinear system: 
\begin{equation}\label{sys}
x_{k+1} = f(x_k, u_k, w_k), 
\end{equation}
where ${x}_k \in \mathbb{R}^n$ is the state at time step $k\in\mathbb{N}$, $u_k \in \mathbb{R}^m$ is the control input, and $w_k \in \mathbb{R}^n$ is the disturbance. 
The control and the disturbance variables are constrained as $u_k \in {\cal U}$, $w_k \in {\cal W}$, $\forall k \in \mathbb{N}$, where 
\begin{align}
{\cal U} &= \{u \in \mathbb{R}^m :  \|u\| \leq u_{\max} \}, \\
{\cal W} &= \{w \in \mathbb{R}^n :  \|w\| \leq w_{\max} \}, 
\end{align}
for given positive constants $u_{\max}$, $w_{\max}$. 
In addition, the state is constrained as $x_k \in {\cal X}$, $\forall k \in \mathbb{N}$, where ${\cal X}$ is a bounded {polygonal} set that can be either a convex or non-convex region. The set ${\cal X}$ represents the \textit{free-space}, in which the state is allowed to move. Inside ${\cal X}$, there exists an initial region ${\cal X}_I \subset {\cal X}$ in which the state is initiated at $k=0$, i.e., $x_0 \in {\cal X}_I$, and a target region ${\cal X}_F \subset {\cal X}$ to which the state aims to move. For simplicity, we assume that the regions are disjoint and are both represented by polytopes. Moreover, let $x_I \in {\cal X}_I$ denote the Chebyshev center \cite{borrelli} of the polytope ${\cal X}_I$. The Chebyshev center is the center of the maximum ball that is included in the polytope and is obtained by solving a linear program (for details, see Section 5.4.5 in \cite{borrelli}). 

For the function $f : \mathbb{R}^n \times \mathbb{R}^m\times \mathbb{R}^n \rightarrow \mathbb{R}^n$, we assume the following: 
\begin{myas}\label{lipschitz_as}
\normalfont
The function $f : \mathbb{R}^n \times \mathbb{R}^m\times \mathbb{R}^n \rightarrow \mathbb{R}^n$ is Lipschitz continuous in $x \in {\cal X}$ and $w \in {\cal W}$, i.e., there exist positive constants $L_{x}$ and $L_{w}$, such that for all $x_1, x_2 \in  {\cal X}$, $w_1, w_2 \in {\cal W}$, and $u \in {\cal U}$, 
\begin{align}
\|f(x_1,  & u, w_1) - f(x_2, u, w_2)\| \notag \\
                   &\leq L_{x} \|x_1 - x_2\| + L_{w} \|w_1 - w_2\|. \label{lipschitz}
\end{align}
\qedwhite
\end{myas}
\subsection{$\delta$-ISS control Lyapunov function} \label{preliminaries}
With respect to the control system \req{sys}, we introduce the following function as the key ingredient to design appropriate control and communication strategies. 

\begin{mydef}\label{cost_def2}
\normalfont 
A smooth function $V: \mathbb{R}^n \times \mathbb{R}^n \rightarrow \mathbb{R}$ is said to be \textit{a $\delta$-ISS control Lyapunov function}, if for all $x, y \in \mathbb{R}^n$, $u \in \mathbb{R}^m$ and $w_1, w_2 \in \mathbb{R}^n$, there exist a smooth function $\kappa: \mathbb{R}^n \times \mathbb{R}^n \times \mathbb{R}^m \rightarrow \mathbb{R}^m$, class ${\cal K}_\infty$-functions $\underline{\alpha}$, $\overline{\alpha}$, $\alpha$, and a class ${\cal K}$-function $\rho$, such that: 
\begin{equation}\label{eq3_valid_cost}
\underline{\alpha} (\|x-y\|) \leq V (x,y) \leq \overline{\alpha} (\|x-y\|) , 
\end{equation}
\begin{align}\label{eq4_valid_cost}
V ( x^+ _{\kappa, w_1},  & y^+ _{u, w_2})- V (x, y) \notag \\
 & \leq - \alpha ( \|x-y\| ) + \rho (\|w_1 - w_2\|), 
\end{align}
where $x^+ _{\kappa, w_1} = f(x, \kappa(x,y, u), w_1)$, $y^+ _{u, w_2} = f(y, u, w_2)$. \qedwhite 
\end{mydef}
The closest notion to \rdef{cost_def2} is a $\delta$-ISS Lyapunov function \cite{incremental,incremental_discrete} that is defined for systems without controls. 
As stated in \cite{incremental,incremental_discrete}, a $\delta$-ISS Lyapunov function is a useful tool to analyze \textit{incremental input-to-state stability}, which captures the contractive behaviors between any pair of the state trajectories. Moreover, the function has been also utilized to obtain finite abstractions for nonlinear control systems in \req{sys}, see, e.g., \cite{girard2009a,girard2010a}. 
In contrast to the afore-cited analysis, in this paper we make use of \rdef{cost_def2} in order to design not only a control strategy such that the resulting state trajectory achieves reachability and safety, but also a {communication strategy} such that the communication reduction is achieved for the NCSs. Note that as shown in \req{eq4_valid_cost}, the state-feedback controller $\kappa$ is applied to only one of the two states (i.e., $x$). The intuition behind here is that we will analyze contractive behaviors between a \textit{closed-loop} state trajectory with the control law $\kappa$ and an \textit{open-loop} state trajectory that will be generated offline, as we will see in later sections. 

Throughout the paper, we assume the existence of the $\delta$-ISS control Lyapunov function:
\begin{myas}\label{cost_as}
\normalfont
For system \req{sys}, there exist smooth functions $V: \mathbb{R}^n \times \mathbb{R}^n \rightarrow \mathbb{R}$ and $\kappa: \mathbb{R}^n \times\mathbb{R}^n \times \mathbb{R}^m \rightarrow \mathbb{R}^m$, such that $V$ is a $\delta$-ISS control Lyapunov function with respect to $\kappa$ satisfying \req{eq3_valid_cost}, \req{eq4_valid_cost}. Moreover, for any $x, y \in \mathbb{R}^n$ and $u\in\mathbb{R}^m$, there exist a class ${\cal K}_\infty$-function $\alpha_u$ and a class ${\cal K}$-function $\rho_u$, such that
\begin{equation}\label{control_assume}
\| \kappa (x, y, u) \| \leq \alpha_u (\|x-y\|) + \rho_u (\|u\|). 
\end{equation}
  \qedwhite
\end{myas}

For example, consider the linear system: $x_{k+1} = f(x_k, u_k, w_k)  = A x_k + Bu_k +w_k$, where the pair $(A, B)$ is assumed to be stabilizable. Let $V (x,y)= \|x-y\|$ be a candidate $\delta$-ISS control Lyapunov function and $\kappa (x, y, u) = -K(x-y) + u$ be the corresponding control law, where $K$ is given such that $A_{cl} := A-BK$ is Hurwitz.  Indeed, the condition \req{eq3_valid_cost} trivially holds and we also have  
\begin{align}
V (x^+ _{\kappa, w_1}, y^+ _{u, w_2}) &= \|A_{cl} (x-y) + w_1 - w_2\| \notag \\
&\leq \sigma_{\max} (A_{cl}) \|x-y\| + \|w_1 - w_2\|, 
\end{align}
so that we have $V (x^+ _{\kappa, w_1}, y^+ _{u, w_2}) - V (x, y) \leq - (1-\sigma_{\max} (A_{cl})) \|x-y\| + \|w_1 - w_2\|$ with $0 < \sigma_{\max} (A_{cl}) < 1$. Thus, the function $V$ is a $\delta$-ISS control Lyapunov function with respect to $\kappa (x, y, u) = -K(x-y) + u$. Moreover, \req{control_assume} holds with $\alpha_u (r)=\sigma_{\max} (K) r$ and $\rho_u (r) = r$. 

\subsection{Overview of the communication strategy}\label{communication_sec}
During the implementation, the plant interacts with the controller over the communication network to update the control inputs in real time. To indicate the communication times, let ${c}_k \in \{0, 1\}$, $k\in\mathbb{N}$ be given by 
\begin{numcases}
{{c}_k = }
	1, \ \ {\rm if\ communication\ occurs\ at }\ k, \\
	0, \ \ {\rm otherwise}. 
\end{numcases}
That is, if $c_k = 1$ the plant transmits the state information $x_k$ to the controller, based on which the control input is updated and transmitted back to the plant. On the other hand, if $c_k = 0$ no communication occurs at $k$. Instead, as we will see later, the plant makes use of a control input that is obtained before the online implementation. 

In this paper, we will present two ways to generate a suitable communication scheduling. The first one is an \textit{offline} approach, in which we preliminary define the communication scheduling $c_k, k\in\mathbb{N}$ before the online implementation. In other words, the communication times are fixed for all state trajectories from ${\cal X}_I$ to ${\cal X}_F$. The offline approach is beneficial in the sense that it does not require any computational effort to generate communication scheduling during online execution. 
On the other hand, the offline communication strategy tends to be conservative, which means that it yields more communication times than the one that is actually (minimally) required to guarantee reachability and safety. In view of this, we further provide an \textit{online} approach as the second communication strategy, in which the controller assigns suitable communication schedulings during online execution. The online communication strategy is given in a self-triggered manner, meaning that for each communication time the controller determines the next communication time based on the state information that is received from the plant. 


\subsection{Problem formulation}
To formulate the problem, we define the \textit{validity} of a state trajectory as follows: 
\begin{mydef}
\normalfont 
For given $x_0 \in {\cal X}_I$, $L\in\mathbb{N}_+$, and disturbance sequence $w_0, w_1,\ldots, w_{L-1} \in {\cal W}$, the trajectory $x_0, x_1, \ldots, x_L$ is called \textit{valid} if the following conditions hold: 
\begin{enumerate}
\item \textit{(Dynamics)}: there exist $u_k \in {\cal U}$, $k \in {\mathbb{N}}_{0:L-1}$, such that $x_{k+1} = f(x_k, u_k, w_k)$, $\forall k \in {\mathbb{N}}_{0:L-1}$; 
\item \textit{(Safety)}: $x_k \in {\cal X}$, $\forall k \in {\mathbb{N}}_{0:L}$; 
\item \textit{(Reachability)}: $x_L \in {\cal X}_F$. \qedwhite
\end{enumerate}
\end{mydef}
That is, the trajectory is valid if there exists a controller such that the state can reach ${\cal X}_F$ in finite time, while at the same time always remaining inside ${\cal X}$ for guaranteeing safety. 
Based on the above, the goal of this paper is to design suitable control and communication strategies, such that the corresponding trajectory becomes valid: 

\begin{mypro}\label{problem}
\normalfont
For a given $x_0 \in {\cal X}_I$, design both control and communication strategies, such that the resulting trajectory is valid for all $w_k \in {\cal W}$, $k \in {\mathbb{N}}$. \qedwhite 
\end{mypro}


\section{Control Strategy}
In this section, we provide a control strategy as a solution to \rpro{problem}. First, we provide an offline procedure to design the control strategy (\rsec{offline_control_sec}). Then, we describe how the control strategy is implemented online (\rsec{control_strategy_sec}). 

\subsection{Offline procedure}\label{offline_control_sec}
In the offline step, we generate a nominal state trajectory from ${\cal X}_{I}$ to ${\cal X}_{F}$. Specifically, we aim to produce a state trajectory $\hat{x}_0, \hat{x}_1, \ldots \hat{x}_L \in {\cal X}$ and the corresponding control $\hat{u}_0, \hat{u}_1, \ldots \hat{u}_{L-1} \in {\cal U}$ for some $L\in {\mathbb{N}}_+$, such that $\hat{x}_0 = x_I$, 
\begin{equation}
\hat{x}_{k+1} = f(\hat{x}_k, \hat{u}_k, 0) \in {\cal X},\ \ \forall k \in {\mathbb{N}}_{0:L-1}, 
\end{equation}
and $\hat{x}_L \in {\cal X}_{F}$. Recall that $x_I$ represents the Chebyshev center of ${\cal X}_I$. Roughly speaking, the trajectory represents a \textit{reference} that the actual state should follow to move from ${\cal X}_{I}$ to ${\cal X}_{F}$. 

So far, numerous techniques have been proposed to generate the reference trajectory as described above. 
Popular ones are the well-known sampling-based algorithms, such as RRT\cite{lavalle1999}, RRT*\cite{karaman2010} and their variants such as g-RRT \cite{dang2008}. Sampling-based algorithms are powerful techniques to find feasible state trajectories even in a complex (non-convex) state-space ${\cal X}$, and have been demonstrated successfully in many control applications, especially in robotics. An alternative method is the cell-decomposition approach \cite{belta2004a,habets2004a}. For example, in \cite{belta2004a} the state-space is decomposed into a set of polytopes, and a piece-wise affine control law is designed for each polytope to steer the state towards the neighboring polytopes. Moreover, we can also utilize optimization-based approaches, such as those solving constrained optimal control problems \cite{collision_free2,ono2011}. 
In view of the many different techniques as illustrated above, the way that the trajectory is derived is beyond the scope of this paper; we can utilize any of the above techniques to obtain the reference state and control trajectories. 


\subsection{Control strategy}\label{control_strategy_sec}
Suppose that we have found reference state and control trajectories $\hat{x}_0, \hat{x}_1, \ldots, \hat{x}_L$, $\hat{u}_0, \hat{u}_1, \ldots, \hat{u}_{L-1}$ in an offline manner according to the procedure presented in the previous subsection. Then, starting from \textit{any} $x_0 \in {\cal X}_I$, the following control strategy is provided during online implementation: for all $k\in {\mathbb{N}}_{0:L-1}$, 
\begin{numcases}
{u_k = }
	\kappa (x_k, \hat{x}_k, \hat{u}_k) , & {\rm if}\ \ $c_k = 1$, \label{closed} \\
	\hat{u}_k, &{\rm if}\ \ $c_k = 0$. \label{open}
\end{numcases}
That is, if the communication is taking place at $k$, the controller applies the state-feedback control law $\kappa$ defined in \ras{cost_as} by using the actual state that is received from the plant. Intuitively, from \rdef{cost_def2} the occurrence of the communication ($c_k = 1$) implies that the error between the actual state $x_k$ and the reference $\hat{x}_k$ potentially becomes smaller. Thus, the occurrence of communication allows the actual state trajectory to track the given reference, which increases the possibility to achieve reachability and safety. On the other hand, if the communication is not given the plant applies the reference $\hat{u}_k$. Although the error may propagate in this case, we can instead reduce the communication frequency by not providing the communication. In the next section, we provide a more quantitative analysis for the above intuition and propose a detailed communication strategy.  
\section{Communication strategy}
Based on the control strategy provided in the previous subsection, we now provide a detailed procedure to generate a communication scheduling $c_k, k\in\mathbb{N}_{0:L-1}$, as well as an implementation algorithm of the communication strategy. 

\subsection{Deriving error propagation model}\label{error_propagation_sec}
Let $v_k \in \mathbb{R}$, $k\in\mathbb{N}_{0:L}$ be given by an error between the actual state and the reference with respect to $V$ at time step $k$, i.e., $v_k = V( x_k,  \hat{x}_k)$, where $V$ is given in \ras{cost_as}. In the following, we describe how this error behaves according to whether communication is given ($c_k = 1$) or not given ($c_k = 0$). For a given $k \in \mathbb{N}_{0:L-1}$, suppose that $c_k = 1$ and the control law in \req{closed} is applied. Let $u_k = \kappa (x_k, \hat{x}_k, \hat{u}_k)$. 
From \req{eq4_valid_cost}, we obtain 
\begin{align*}
V ( x_{k+1}, & \hat{x}_{k+1}) - V(x_k, \hat{x}_k) \notag  \\
& \leq - \alpha_2 (V(x_k, \hat{x}_k)) + \rho (\|w_k\|), 
\end{align*}
where $x_{k+1} = f(x_k, u_k, w_k)$, $\hat{x}_{k+1} = f(\hat{x}_k, \hat{u}_k, 0)$ and $\alpha_2 = \alpha \circ \overline{\alpha}^{-1} $. Thus, we obtain
\begin{align*}
v_{k+1} \leq &  ({\rm Id} - \alpha_2) (v_k) + \rho (\|w_k\|), 
\end{align*}
where ${\rm Id}:\mathbb{R}\rightarrow \mathbb{R}$ denotes the identity function. 
Without loss of generality, we assume that the function $({\rm Id} - \alpha_2)$ is a class ${\cal K}_\infty$-function\footnote{This is due to the fact that for any ${\cal K}_\infty$-function $\alpha_2$, there exists a class ${\cal K}_\infty$-function $\hat{\alpha}_2$ such that: (i) $\hat{\alpha}_2 (r) \leq \alpha_2 (r)$, $\forall r \geq 0$; (ii) ${\rm Id} - \hat{\alpha}_2$ is a class ${\cal K}_\infty$-function, see \cite{iss}. }. In addition to the error propagation of the states, it is required that the control input satisfies the constraint, i.e., $u_k \in {\cal U}$. To derive the condition for this, observe that from \req{control_assume} we obtain 
\begin{align}\label{control_input_upper}
\| \kappa (x_k, \hat{x}_k, \hat{u}_k) \| \leq \alpha_u \circ \underline{\alpha}^{-1} (v_k ) + \rho_u (\|\hat{u}_k\|), 
\end{align}
where we have used $\|x_{k} -\hat{x}_{k}\| \leq \underline{\alpha}^{-1} (V(x_k, \hat{x}_k))$ from \req{eq3_valid_cost}. Thus, $u_k = \kappa (x_k, \hat{x}_k, \hat{u}_k) \in {\cal U}$ if $\alpha_u \circ \underline{\alpha}^{-1} (v_k ) + \rho_u (\|\hat{u}_k\|) \leq u_{\max}$. 

Suppose now that $c_k = 0$ and the control law in \req{open} is applied. Let $u_k = \hat{u}_k$. 
From the Lipschitz continuity in \ras{lipschitz_as}, it follows that 
\begin{equation*}
\|x_{k+1} - \hat{x}_{k+1} \| \leq L_x \|x_k - \hat{x}_k\| + L_w \|w_k\|, 
\end{equation*}
where $x_{k+1} = f(x_k, u_k, w_k)$ and $\hat{x}_{k+1} = f(\hat{x}_k, \hat{u}_k, 0)$ (with $u_k = \hat{u}_k$). Moreover, from \req{eq3_valid_cost} we obtain $\|x_{k} -\hat{x}_{k}\| \leq \underline{\alpha}^{-1} (V(x_k, \hat{x}_k))$ and $\overline{\alpha}^{-1} (V(x_{k+1}, \hat{x}_{k+1})) \leq \|x_{k+1} -\hat{x}_{k+1}\|$. 
Thus, we obtain 
\begin{align*}
v_{k+1} &\leq \overline{\alpha} \left( L_x \underline{\alpha}^{-1} (v_k) + L_w \|w_k\| \right).  
\end{align*}
Note that the input constraint is satisfied for this case, i.e., $u_k = \hat{u}_k \in {\cal U}$.  
Consequently, for both $c_k = 1$ and $0$, we obtain
\begin{equation}\label{vnext}
{v}_{k+1} \leq g(v_k, c_k, \|w_k\|), 
\end{equation}
where the function $g: \mathbb{R} \times \{0, 1\} \times \mathbb{R} \rightarrow \mathbb{R}$ is defined by
\begin{align}
g(v, c, \|w\|) & =\ c \left( ({\rm Id} - \alpha_2) (v) + \rho (\|w\|) \right ) \notag \\
                          & + (1-c) \left\{ \overline{\alpha} \left(L_x \underline{\alpha}^{-1} (v) + L_w \|w\|\right) \right \}. \label{gfunc}
\end{align}
The inequality \req{vnext} indicates that if the communication is given at $k$ (i.e., $c_k = 1$), the error between the actual state and the reference potentially becomes smaller. Indeed, if $v_k$ is large enough to satisfy $\alpha_2 (v_k) > \rho (w_{\max})$, it holds that $v_{k+1} \leq v_k - \alpha_2 (v_k) + \rho (\|w_k\|) < v_k$ and the error gets strictly smaller at the next time. On the other hand, the error may grow according to \req{vnext} (with $c_k = 0$) if the communication is not given. Thus, we can quantitatively evaluate the propagation of the error according to the relation given in \req{vnext}. 
Note that since $\overline{\alpha}$, $\underline{\alpha}$, ${\rm Id} - \alpha_2$, and $\rho$ are class ${\cal K}_\infty$ (or ${\cal K}$)-functions, the function $g$ is \textit{monotone}\cite{monotone} with respect to $v \in \mathbb{R}$ and $\|w\|\in \mathbb{R}$, i.e., for any $c\in\{0,1\}$, $v, v' \in\mathbb{R}$ with $v\leq v'$, and $w, w' \in {\cal W}$ with $\|w\|\leq \|w'\|$, it holds that 
\begin{align}\label{monotone}
g(v, c, \|w\|) \leq g(v', c, \|w\|) \leq g(v', c, \|w'\|). 
\end{align}
As will be shown below, the monotonicity property plays an important role to derive suitable conditions to guarantee  the validity of the state trajectory. 

For given $c_k \in \{0,1\}$, $k\in {\mathbb{N}}_{0:L-1}$ and $\overline{v}_0 \in \mathbb{R}$, let $\overline{v}_{k} \in \mathbb{N}$, $k\in {\mathbb{N}}_{0:L}$ be recursively given by 
\begin{equation}\label{vbar}
\overline{v}_{k+1} = g (\overline{v}_k, c_k, w_{\max}). 
\end{equation}
That is, $\overline{v}_{k}$ represents the \textit{upper bound} of $v_k$ by setting the disturbance sequence as the maximal (worst case) one $\|w_k\| = w_{\max}$ and considering the equality in \req{vnext} for all $k\in\mathbb{N}_{0:L-1}$. In addition, let $v_{k, \max} \in \mathbb{R}_+$, $k\in {\mathbb{N}}_{0:L}$ be given by 
\begin{align}\label{emax}
v_{k, \max} = {{\max}} \{ \varepsilon \in \mathbb{R}: {\cal B}_\varepsilon (\hat{x}_k) \subseteq {\cal X}\}, 
\end{align}
where ${\cal B}_\varepsilon (\hat{x}) = \{x\in\mathbb{R}^n \ |\ V(x, \hat{x}) \leq \varepsilon\}$. The set ${\cal B}_\varepsilon (\hat{x})$ indicates the set of all states around $\hat{x}$ such that the error value is less than $\varepsilon$. From \req{emax} it follows that 
\begin{equation}\label{safety_eq}
V(x_k, \hat{x}_k) \leq v_{k,\max} \implies x_k \in {\cal B}_{v_{k,\max}} (\hat{x}_k) \subseteq {\cal X}. 
\end{equation}
Thus, $v_{k, \max}$ represents the maximum value of the error at $k$ such that the actual value of the state at $k$ guarantees safety. Finally, let $v_{init}, v_{final} \in \mathbb{R}$ be given by 
\begin{align}
v_{init} &= \min \{\varepsilon\in\mathbb{R}: {\cal X}_I \subseteq {\cal B}_{\varepsilon} (\hat{x}_0)\} \label{vinit} 
\end{align}
\begin{align}
v_{final} &= \max \{\varepsilon\in\mathbb{R}:  {\cal B}_{\varepsilon} (\hat{x}_L) \subseteq {\cal X}_F\}.  \label{vfinal}
\end{align}
Based on the above notations, we obtain the following result: 


\begin{mylem}\label{control_strategy_result}
\normalfont
Suppose that the communication scheduling $c_k \in \{0, 1\}$, $k\in {\mathbb{N}}_{0:L-1}$ is designed such that:  
\renewcommand{\labelenumi}{(C.\arabic{enumi})}
\begin{enumerate}
\item $\overline{v}_0 = v_{init}$;
\item $\overline{v}_k \leq v_{k,\max}$, $\forall k \in \mathbb{N}_{1:L}$; 
\item $\overline{v}_L \leq v_{final}$; 
\item $c_k = 1 \implies \alpha_u \circ \underline{\alpha}^{-1} (\overline{v}_k ) + \rho_u (\|\hat{u}_k\|) \leq u_{\max}$, \\ $\forall k \in \mathbb{N}_{0:L-1}$, 
\end{enumerate}
where $\overline{v}_1, \ldots, \overline{v}_L$ are computed according to \req{vbar}. Then, for \textit{any} $x_0 \in {\cal X}_I$ and $w_0, w_1, \ldots, w_{L-1} \in {\cal W}$, the state trajectory $x_0, x_1, \ldots, x_L$ becomes valid by applying the control strategy in \req{closed} and \req{open}. \qedwhite
\end{mylem}
\begin{proof}
Suppose that $c_k \in \{0, 1\}$, $k\in {\mathbb{N}}_{0:L-1}$ is designed such that the conditions (C.1)--(C.4) are fulfilled. For any $x_0 \in {\cal X}_I$, let $x_0, x_1, \ldots , x_L$ and $u_0, u_1, \ldots, u_{L-1}$ be the (actual) state and the corresponding control trajectories according to \req{closed} and \req{open}. 
From \req{vinit}, it holds that $x_0 \in {\cal X}_I\subseteq {\cal B}_{\bar{v}_{0}} (\hat{x}_0)$ and so we have $v_0 = V(x_0, \hat{x}_0) \leq \bar{v}_{0}$. Thus, for any $w_0 \in {\cal W}$ we obtain 
\begin{equation*}
v_1 \leq g(v_0, c_0, \|w_0\|) \leq g(\overline{v}_0, c_0, w_{\max}) = \bar{v}_{1}, 
\end{equation*}
where we have used the monotonicity property in \req{monotone}. Since $v_1 \leq \overline{v}_{1}$, we obtain 
\begin{equation*}
v_2 \leq g(v_1, c_1, \|w_1\|) \leq g(\overline{v}_1, c_1, w_{\max}) = \bar{v}_{2}, 
\end{equation*}
for any $w_1 \in {\cal W}$. 
Using the same procedure, we recursively obtain $v_k \leq \overline{v}_{k}$, $\forall k\in\mathbb{N}_{0:L}$ for any $w_0, w_1, \ldots, w_{L-1} \in {\cal W}$. 
From (C.2), it then holds that $v_k \leq v_{k, \max}$, $\forall k\in\mathbb{N}_{1:L}$, which means from \req{safety_eq} that 
\begin{equation}
x_k \in {\cal B}_{v_{k, \max}} (\hat{x}_k) \subseteq {\cal X},
\end{equation}
$\forall k\in\mathbb{N}_{1:L}$. Thus, the state trajectory guarantees safety. Moreover, it follows from (C.3) and \req{vfinal} that 
$x_L \in {\cal B}_{v_{final}} \subseteq {\cal X}_F$, which means that the state trajectory guarantees reachability. 
In addition, from (C.4) and \req{control_input_upper}, $c_k = 1$ implies that 
\begin{align}
\|u_k\| &\leq \alpha_u \circ \underline{\alpha}^{-1} (v_k) + \rho_u (\|\hat{u}_k\| ) \notag \\
& \leq \alpha_u \circ \underline{\alpha}^{-1} (\overline{v}_k) + \rho_u (\|\hat{u}_k\| ) \leq u_{\max}, 
\end{align} 
and $c_k = 0$ implies $\|u_k\| = \|\hat{u}_k\| \leq u_{\max}$. 
Thus, it holds that $u_k \in {\cal U}$, $\forall k \in \mathbb{N}_{0:L-1}$. Therefore, it is shown that the trajectory $x_0, x_1, \ldots, x_L$ becomes valid. Since this holds for any $x_0 \in {\cal X}_I$ and $w_0, w_1, \ldots, w_{L-1} \in {\cal W}$ the proof is complete. 
\end{proof}

\rlem{control_strategy_result} indicates that if the communication scheduling $c_k$, $k\in\mathbb{N}_{0:L-1}$ is given such that the sequence of errors with $\overline{v}_0 = v_{init}$ becomes \textit{small enough} to satisfy (C.2) and (C.3), as well as that the condition (C.4) holds in order to satisfy the input constraint, then every state trajectory starting from $x_0 \in {\cal X}_I$ becomes valid. 

\begin{myrem}[On the computation of $v_{k, \max}$]\label{compute_vmax}
\normalfont
Note that in order to check (C.1)--(C.3) one is required to compute $v_{k,\max}$, $\forall k \in \mathbb{N}_{1:L}$ (as well as $v_{init}, v_{final}$). 
For the linear case, a $\delta$-ISS control Lyapunov function can be chosen as the error norm $V(x,y) = \|x-y\|$ (see \rsec{preliminaries}) and thus the set ${\cal B}_{\varepsilon} (\hat{x})$ is defined as a ball with center $\hat{x}$ and radius $\varepsilon$. Let $\overline{\cal B}_{\varepsilon} (\hat{x})$ be a polytope with ${\cal B}_{\varepsilon} (\hat{x})\subseteq \overline{\cal B}_{\varepsilon} (\hat{x})$. The set $\overline{\cal B}_{\varepsilon} (\hat{x})$ can be of any shape (e.g., square, hexagon) but is selected to include the ball ${\cal B}_{\varepsilon} (\hat{x})$. Since \textit{any} polygonal set ${\cal X}$ can be cell-decomposed as ${\cal X} = \bigcup^{N_x} _{n=1} {\cal X}_n$ ($N_x$ denotes the number of polytopes obtained by the decomposition), it holds that ${\cal B}_{\varepsilon} (\hat{x}) \subseteq {\cal X}$ if 
\begin{equation*}
\bigcup^{N_x} _{n=1} {\cal X}_{\varepsilon, n} = \overline{\cal B}_{\varepsilon} (\hat{x}), 
\end{equation*}
where ${\cal X}_{\varepsilon, n} = {\cal X}_n \cap \overline{\cal B}_{\varepsilon} (\hat{x})$. Since ${\cal X}_n$ and $\overline{\cal B}_{\varepsilon} (\hat{x})$ are both polytopes, ${\cal X}_{\varepsilon, n}$ is a polytope that can be computed by vertex operations. 
Thus, $v_{k, \max}$ can be obtained (under-approximated) by searching the maximum value of $\varepsilon$ with the property that the union of all ${\cal X}_{\varepsilon, n}$, $n \in \mathbb{N}_{1:N_x}$ is equal to $\overline{\cal B}_{\varepsilon}$. For general nonlinear systems, however, it may be difficult to compute $v_{k, \max}$ since the function $V(x, y)$ is in general not given by the error norm. In this case, we can make use of property in \req{eq3_valid_cost} in the following way. For a given $\varepsilon>0$, let ${\cal B}_{\alpha, \varepsilon} (\hat{x}_k)$ be the ball set characterized as ${\cal B}_{\alpha, \varepsilon} (\hat{x}_k)= \{x_k\in\mathbb{R}^n :  \|x - \hat{x}_k\| \leq \underline{\alpha}^{-1}  (\varepsilon)\}$, where the function $\underline{\alpha}$ is defined in \req{eq3_valid_cost}. Since $V(x_k, \hat{x}_k) \leq \varepsilon \implies \underline{\alpha} ( \|x_k - \hat{x}_k\|) \leq \varepsilon$ for any $\varepsilon>0$, it holds that ${\cal B}_{\varepsilon} (\hat{x}_k) \subseteq {\cal B}_{\alpha, \varepsilon} (\hat{x}_k)$ for any $\varepsilon>0$. Therefore, $v_{k, \max}$ can be under-approximated by searching the maximum value of $\varepsilon$ with ${\cal B}_{\alpha, \varepsilon} (\hat{x}_k) \subseteq {\cal X}$, which can be done by employing the same procedure as for the linear case described above.
\qedwhite 

\end{myrem}

\subsection{A naive approach to generate communication scheduling}\label{naive_approach_sec}

From \rlem{control_strategy_result}, if we assign $c_k$, $\forall k\in\mathbb{N}_{0:L}$ such that the conditions (C.1)--(C.4) are satisfied, the resulting state trajectory guarantees both safety and reachability. At the same time, we can achieve the communication reduction as much as possible by finding the minimum number of communication instants (i.e., $\min \sum^{L-1} _{k=0} c_k$). 
In what follows, we present ideas to achieve such communication scheduling. As a starting point, this subsection provides a \textit{naive} approach as a motivation to derive our proposed solution in the next subsections. 

\begin{figure}[t]
  \begin{center}
   \includegraphics[width=7cm]{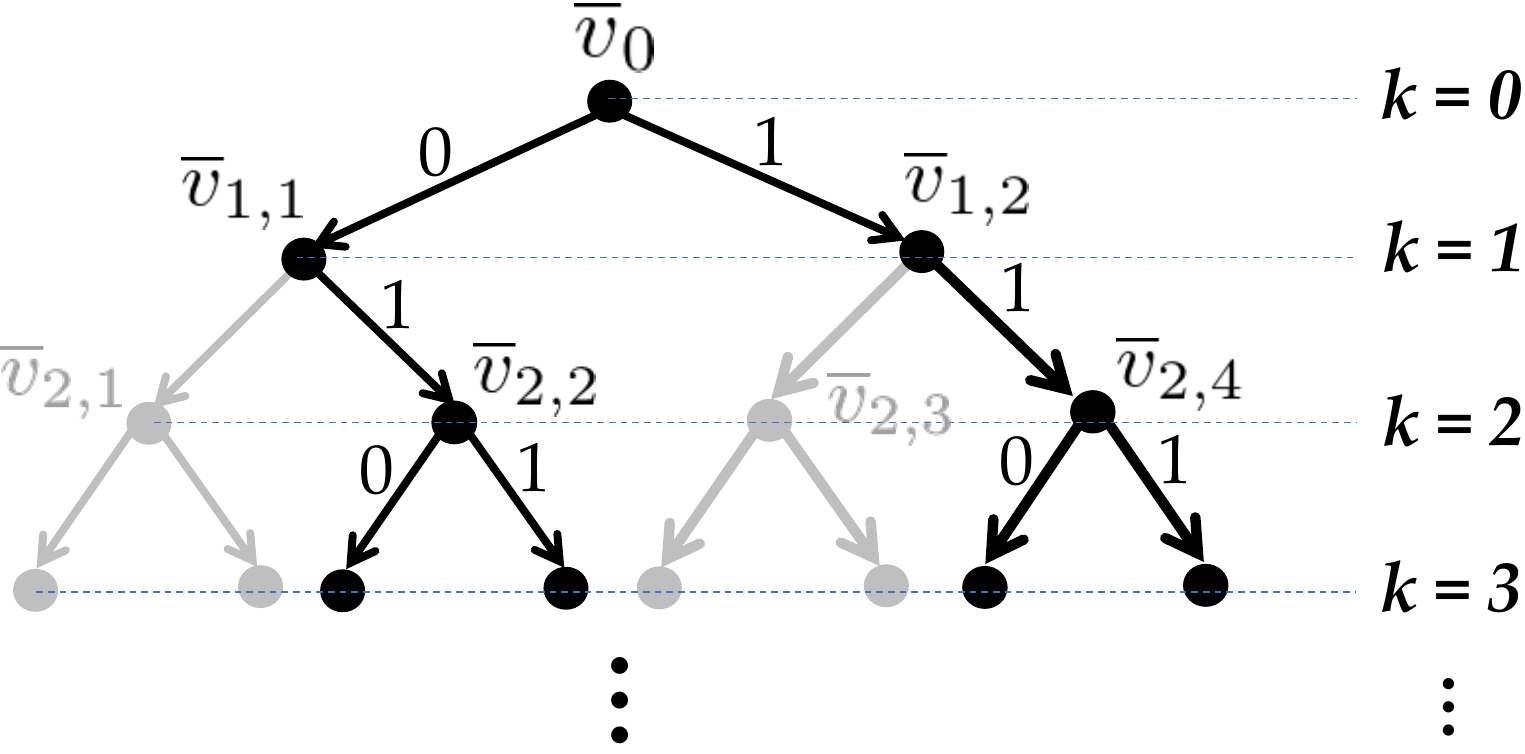}
   \caption{Illustration of generating a binary tree as a naive approach to find the communication scheduling. In the figure, gray nodes and edges are eliminated since the safety (reachability) conditions are violated. For instance, the node with $\overline{v}_{2,1}$ is eliminated since $\overline{v}_{2,1} \leq v_{2,\max}$ does not hold. } 
   \label{tree_example}
  \end{center}
 \end{figure}

Suppose that we obtain $\hat{x}_0, \hat{x}_1, \ldots, \hat{x}_L$ and $\hat{u}_0, \hat{u}_1, \ldots, \hat{u}_{L-1}$ as, respectively, the reference state and control trajectories as described in \rsec{offline_control_sec}, as well as $v_{k, \max}$, $\forall k\in {\mathbb{N}}_{0:L}$ as described in \rsec{error_propagation_sec} (see in particular \rrem{compute_vmax}). The most straightforward approach to obtain the desired communication scheduling $c_k$, $\forall k\in\mathbb{N}_{0:L}$ may be to consider all possible (feasible) communication schedulings satisfying all conditions (C.1)--(C.4), and then find the optimal one providing the minimum number of communication instants. The overview of this approach is illustrated in \rfig{tree_example}. As shown in the figure, the problem is translated into the construction of a \textit{binary tree} with the depth $L$. That is, starting from $\bar{v}_0 = v_{init}$ we first compute the next error for the two cases, according to whether the communication is given ($\bar{v}_{1,1} = g (\overline{v}_0, 1, w_{\max})$) or not given ($\bar{v}_{1,2} = g (\overline{v}_0, 0, w_{\max})$). For each case, we check if the corresponding error value is below the safety (or reachability if $L=1$) bounds, i.e., check if $\bar{v}_{1,1} \leq v_{1,\max}$ and $\bar{v}_{1,2} \leq v_{1,\max}$ (or $\bar{v}_{1,1} \leq v_{final}$ and $\bar{v}_{1,2} \leq v_{final}$ if $L=1$), as well as check if it satisfies the input constraint according to (C.4). If the conditions are satisfied, we add the node and the corresponding edge to the tree. If the condition does not hold, no further nodes and edges are added. The above procedure is iterated until the terminal time step $k = L$ is reached. Once the tree has been constructed, we seek the optimal path from the initial node to the ones at $k=L$, which provides the minimum number of communication instants. 

By using the above procedure, we can find a communication scheduling such that the conditions (C.1)--(C.4) are rigorously satisfied, and thus the resulting state trajectory achieves reachability and safety. However, the main drawback of the above procedure is its computational complexity; the number of total nodes for the constructed binary tree (as well as the number of feasible paths) for the worst case is $2^L$. Thus, the number of total nodes grows exponentially with respect to the total time steps, which makes the construction of a binary tree intractable. 

Motivated by the above issue, in the following subsections we provide an alternative approach that makes the problem of finding the communication scheduling \textit{scale well} with respect to the time step $L$. In particular, we propose to construct a \textit{symbolic error system}, which represents an {abstracted} behavior of the upper bound of the error propagation model in \req{vbar}. 
The symbolic system is \textit{abstracted} in the sense that it deals with only a finite number of non-negative reals that are selected from the domain of $\mathbb{R}$, in contrast to the original model in \req{vbar} that is defined over \textit{all} non-negative reals in $\mathbb{R}$. 
The symbolic system is constructed by making use of the monotonicity property in \req{monotone}, and it allows us to generate desired communication schedulings with the computational complexity being much more tractable than the naive approach. 


\subsection{Abstracting the behavior of the error propagation model}\label{abstraction_sec}
In this subsection we provide an approach to construct a symbolic model representing an abstracted behavior of \req{vbar}. To this end, we first \textit{partition} the domain of $\mathbb{R}$ into a {finite} number of segments. 
That is, for given $M\in\mathbb{N}_+$ with $M\geq2$ and $\overline{\nu} \in \mathbb{R}_+$, define a set of scalars $0< \nu_1 < \nu_2 < \cdots <\nu_{M}$ given by 
\begin{equation}\label{rhoseq}
\nu_m = \overline{\nu} m/ (M-1) , \ \ \forall m \in \mathbb{N}_{1:M-1}
\end{equation}
and $\nu_M = \infty$. Here, $\overline{\nu} = \nu_{M-1}$ represents the maximum finite value among the set of scalars $\nu_1, \ldots, \nu_{M-1}$, and $M$ represents the number of cells in the partition of the domain $\mathbb{R}$. How these parameters should be given is described later in this section. 
The sequence $\nu_1$, $\ldots$, $\nu_M$ represents the upper bound of the errors that will be treated to generate the communication schedulings. Namely, instead of using the original model \req{vbar} that is defined over \textit{all} non-negative reals in $\mathbb{R}$, we use here an {abstracted} model that consists of only a \textit{finite} number of positive reals $\{\nu_1, \ldots, \nu_M\}$. We refer to this abstracted model as the \textit{symbolic error system}, which is formally defined below: 

\begin{mydef}\label{transition_relation}
\normalfont
A \textit{symbolic error system} ${\cal T}$ is a tuple 
\begin{equation}
{\cal T} = ( S, \gamma, \delta), 
\end{equation}
where 
\begin{itemize}
\item $S = \{s_1, s_2, \ldots, s_M \}$ is a set of symbols; 
\item $\gamma : S \rightarrow \mathbb{R}$ is a labeling function given by $\gamma (s_i) = \nu_i$, $\forall i \in \mathbb{N}_{1:M}$; 
\item $\delta \ \subseteq S \times \{0,1\} \times S$ is a transition relation defined as follows: for given $s_i \in S$ and $c\in\{0, 1\}$, let $s_j \in S$ be given by 
\begin{align}
s_j &= \underset{s \in S}{{\rm arg\ min}}\ \gamma (s),\  {\rm s.t.}\ g (\gamma (s_i), c, w_{\max}) \leq \gamma (s).  \label{gcond} \hspace{-0.15cm} 
\end{align} 
Then, $(s_i, c, s_j) \in \delta$. \qedwhite
\end{itemize}
\end{mydef}

The system ${\cal T}$ mainly consists of a finite set of symbols $S = \{s_1, \ldots, s_M\}$ and their transitions defined in $\delta$. Each symbol $s_i \in S$ is related to the scalar $\nu_i$ through the mapping $\gamma$. In other words, the symbol $s_i$ indicates that the upper bound of the error is $\nu_i$. The transition relation $\delta$ is defined by solving \req{gcond}. Here, $g(\gamma(s_i), c, w_{\max})$ represents the upper bound of the error at the next time from the one associated with $s_i$ with (or without) the occurrence of communication indicated by $c$. Thus, the next symbol to be transitioned is determined by taking the closest upper bound to $g(\gamma(s_i), c, w_{\max})$. 
Roughly speaking, the transition indicates that the upper bound of the error becomes smaller (or larger) by providing (or not providing) the communication. 
For example, $(s_2, c, s_1) \in \delta$ with $c=1$ indicates that the upper bound of the error \textit{decreases} from $\nu_2$ to $\nu_1$ by the occurrence of communication. On the other hand, $(s_1, c, s_2) \in \delta$ with $c=0$ indicates that the upper bound of the error grows from $\nu_{1}$ to $\nu_{2}$ by not providing the communication. 
Note that the transition system ${\cal T}$ is \textit{deterministic} and \textit{non-blocking} (see, e.g., \cite{baier}), which means that for every $s \in S$ and $c \in \{0, 1\}$ there exists one transition from $s$. \\

\noindent
\textit{(Example~1):} 
Consider the linear system $x_{k+1} = f(x_k, u_k, w_k)  = A x_k + Bu_k +w_k$, where the pair $(A, B)$ is assumed to be stabilizable. As described in \rsec{preliminaries}, define $V(x,y) = \|x-y\|$ as the $\delta$-ISS control Lyapunov function with respect to $\kappa (x, y, u) = -K(x-y) + u$, where $K$ is given such that $A_{cl} = A-BK$ is Hurwitz. 
The corresponding upper bound of the error model is given by \req{vbar} with 
\begin{align}
g(v, c, \|w\|) = & c  \left(\sigma_{\max} (A_{cl}) + \|w\| \right) \notag \\
                  & + (1-c) \left(\sigma_{\max} (A) + \|w\| \right). 
\end{align}
Suppose that we have $\sigma_{\max} (A_{cl})  = 0.6$, $\sigma_{\max} (A)  = 1.2$, and $w_{\max} = 0.1$, and the parameters for the partition of the domain $\mathbb{R}$ are $\overline{\nu} = 5$, $M = 6$ (i.e., $\nu_m = m$, $m\in\mathbb{N}_{1:5}$ and $\nu_6 = \infty$).  The symbolic error system ${\cal T}$ is illustrated as a graph in \rfig{error_example}. For example, we have $g (\gamma (s_5), 1, w_{\max}) = 5 \cdot 0.6 + 0.1 = 3.1 < 4$ and so $(s_5, 1, s_4) \in \delta$. Also, we have $g (\gamma (s_5), 0, w_{\max}) = 5 \cdot 1.2 + 0.1 = 6.1 >5$ and so $(s_5, 0, s_6) \in \delta$. Note that we have $(s_2, 1, s_2) \in \delta$ since $g (\gamma (s_2), 1, w_{\max}) = 2 \cdot 0.6 + 0.1 = 1.3 < 2$. This means that the rate of decrease of the error is not large enough to transition to $s_1$. Thus, the state that the upper bound of the error is $1$ (i.e., $s_1$) can never be reached in ${\cal T}$. In the original system \req{vbar}, on the other hand, it holds that $\overline{v}_{k+1}\leq 1.0$ for all $\overline{v}_{k}\leq 1.5$, which means that there exists some $\overline{v}_k$ that can become smaller than $1$ at the next time. 
Thus, the symbolic system provides a somewhat conservative behavior in the sense that the error cannot be further reduced in comparison to the original system in \req{vbar}. 
\qedwhite

\begin{figure}[t]
  \begin{center}
   \includegraphics[width=6cm]{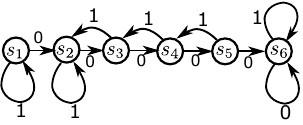}
   \caption{Symbolic error system obtained for Example~1. } 
   \label{error_example}
  \end{center}
 \end{figure}

\begin{myrem}[On the selection of $\overline{\nu}$]\label{select_param}
\normalfont
When constructing ${\cal T}$, one needs to define the parameters $\overline{\nu}$ and $M$ in order to characterize the partition of $\mathbb{R}$. In this remark, we provide a way to determine the parameter $\overline{\nu}$; regarding $M$, please refer to \rsec{computation_sec} for details as well as the illustrative analysis in the simulation section (\rsec{sim_nonlinear_sec}). Let us recall that we compute $v_{k,\max}$, $\forall k\in\mathbb{N}_{0:L}$ in \req{emax}.  Here, each $v_{k,\max}$ represents the \textit{maximum allowable error} that can be taken at $k$, which means that the error does not take values more than $v_{k,\max}$. Since $\overline{\nu}$ determines the maximum finite error value associated with the symbol in ${\cal T}$, one way to select $\overline{\nu}$ is to take the maximum value among $v_{k,\max}$,  $k\in\mathbb{N}_{0:L}$, i.e., 
\begin{equation}
\overline{\nu} = \underset{k\in\mathbb{N}_{0:L}}{\max}\ v_{k,\max}. 
\end{equation}
Note that since ${\cal X}$ is a bounded polygonal set, $v_{k,\max}$, $k\in\mathbb{N}_{0:L}$ as well as $\overline{\nu}$ are finite. 
\qedwhite 
\end{myrem}

\subsection{Generating the communication plan: an offline approach}\label{offline_sec}
Based on the symbolic system provided in the previous subsection, we now provide a framework to generate the desired communication scheduling. As described in \rsec{communication_sec}, we present both an offline and an online framework to generate the communication scheduling; in this subsection we derive the former approach. In order to assign a suitable communication scheduling in an offline manner, we further construct the following symbolic system: 
\begin{mydef}\label{transition_relation2}
\normalfont
A \textit{timed} symbolic error system is a tuple \begin{equation}
{\cal T}_A = ( S_A, \delta_A, s_{A,init}, S_{A,final}), 
\end{equation}
where 
\begin{itemize}
\item $S_A = S \times \mathbb{N}_{0:L}$ is a set of symbols; 
\item $\delta_A \subseteq S_A \times \{0,1\} \times S_A$ is a transition relation, and $((s_i, k), c, (s_j, k+1)) \in \delta_A$ for all $k\in\mathbb{N}_{0:L-1}$ if the following conditions hold: 
\renewcommand{\labelenumi}{(D.\arabic{enumi})}
\begin{enumerate}
\item $(s_i, c, s_j) \in \delta$; 
\item $\gamma (s_i) \leq v_{k,\max}$; 
\item $\gamma (s_j) \leq v_{k+1,\max}$;
\item $c =1 \implies \alpha_u \circ \underline{\alpha}^{-1} (\gamma (s_i)) + \rho_u (\|\hat{u}_k\|) \leq u_{\max}$;  
\end{enumerate}
\item $s_{A, init} = (s_{init}, 0) \in S_A$ is an initial state, where $s_{init} \in S$ is given by 
\begin{equation}
s_{init} = \underset{s\in S}{{\rm arg\ min}}\ \gamma (s),\  {\rm s.t.}\ {\cal X}_I \subseteq {\cal B}_{\gamma(s)} (\hat{x}_0); 
\end{equation}
\item $S_{A, final} \subset S_A$ is a set of terminal states given by 
\begin{align}
S_{A, final} &= \left\{ (s, L) \in S_A : {\cal B}_{\gamma(s)} (\hat{x}_L) \subseteq {\cal X}_F \right \}. \label{sfinal}
\end{align} 
\qedwhite
\end{itemize}
\end{mydef}

A timed symbolic system ${\cal T}_A$ is provided as an extension to the original one ${\cal T}$, in the sense that the time step domain in the symbolic states as well as an initial state and a set of terminal states are additionally defined in the symbolic system. As shown in the definition, a transition is allowed from $(s_i, k)$ to $(s_j, k+1)$ only if the conditions (D.1)--(D.4) are satisfied. Condition (D.1) indicates that the transition from $s_i$ to $s_j$ is allowed in the original transition system ${\cal T}$. Conditions (D.2), (D.3) indicate that the upper bounds of the error associated with $s_i$ and $s_j$ are respectively below the safety bounds $v_{k,\max}$ and $v_{k+1, \max}$ that are defined in \req{emax}. Condition (D.4) is essentially required in order for the control input to satisfy the input constraint. Overall, the procedure to derive $\delta_A$ is presented in \ralg{delta_construct_alg}. As shown in the algorithm, for each time step and for each transition in ${\cal T}$, we check if the conditions (D.2) -- (D.4) are satisfied. Only if the conditions are satisfied, we add the corresponding transition in $\delta_A$. As we will see in the next subsection, the complexity for implementing \ralg{delta_construct_alg} is shown to be \textit{linear} with respect to $L$, which is thus much more tractable than the construction of a binary tree for the naive approach described in \rsec{naive_approach_sec}. 

\begin{algorithm}[t]\label{delta_construct_alg}
{\small 
\SetKwInOut{Input}{input}
\SetKwInOut{Output}{output}
\Input{${\cal T}$ (transition system), $v_{k,\max}$, $\forall k\in\mathbb{N}_{0:L}$,\\ $\hat{u}_{k}$, $\forall k\in\mathbb{N}_{0:L-1}$}
\Output{$\delta_A$ (transition relation for ${\cal T}_A$)} 
set $\delta_A = \{ \}$ (initialization); \\
    \For{ $0\leq k \leq L-1$ }{
      \For{ each $(s_i, c, s_j) \in \delta$}{
      	 \If {(D.2)--(D.4) are satisfied} { 
      		$\delta_A \leftarrow \delta_A \cup ((s_i,k), c, (s_j, k+1))$;
      				}
    				}
           }
    \caption{Derivation of $\delta_A$. } 
    }
\end{algorithm}

We proceed to define the accepting run for ${\cal T}_A$ by recalling the standard definitions from automata theory (see, e.g., \cite{baier}): 
\begin{mydef}\label{accepting_def}
\normalfont
For given $c_k \in \{0,1\}$, $\forall k\in\mathbb{N}_{0:L-1}$, the sequence $(s(0), 0), (s(1), 1), \ldots, (s(L), L)$ is called an \textit{accepting run} for $c_k$, $k\in\mathbb{N}_{0:L-1}$ in ${\cal T}_A$, if $(s(0), 0) = s_{A, init}$, $((s(k), k), c_k, (s(k+1), k+1)) \in \delta_A$, $\forall k\in\mathbb{N}_{0:L-1}$, and $(s(L), L) \in S_{A, final}$. Moreover, $c_k \in \{0,1\}$, $\forall k\in\mathbb{N}_{0:L-1}$ is called \textit{accepted by ${\cal T}_A$}, if there exists an accepting run for $c_k$, $k\in\mathbb{N}_{0:L-1}$. \qedwhite 
\end{mydef}

Based on the above definitions, we obtain the following result: 
\begin{mythm}\label{main_result}
\normalfont
Suppose that the communication scheduling $c_k \in \{0, 1\}$, $k\in {\mathbb{N}}_{0:L-1}$ is designed such that it is accepted by ${\cal T}_A$. 
Then, for \textit{any} $x_0 \in {\cal X}_I$ and $w_0, w_1, \ldots, w_{L-1} \in {\cal W}$, the resulting state trajectory $x_0, x_1, \ldots, x_L$ is valid by applying the control strategy in \req{closed} and \req{open}. \qedwhite
\end{mythm}

\begin{proof}
Suppose that $c_k$, $k\in \mathbb{N}_{0:L-1}$ is accepted by ${\cal T}_A$, and let $(s(0), 0), (s(1), 1), \ldots, (s(L), L)$ be the accepting run for $c_k$, $k\in \mathbb{N}_{0:L-1}$. Based on $c_k$, $k\in \mathbb{N}_{0:L-1}$ and for a given $\overline{v}_0 = v_{init}$, let $\overline{v}_1$, $\cdots$, $\overline{v}_L$ be the sequence computed according to \req{vbar}. 
In addition, let $\widetilde{v}_k = \gamma (s(k))$, $\forall k\in\mathbb{N}_{0:L}$. From \req{vinit}, we obtain 
\begin{align*}
\overline{v}_0 &= \min \{\varepsilon\in\mathbb{R}: {\cal X}_I \subseteq {\cal B}_{\varepsilon} (\hat{x}_0)\} \\
                   &\leq \min \{\gamma (s) : s\in S, {\cal X}_I \subseteq {\cal B}_{\gamma (s)} (\hat{x}_0) \} \\
                   & = \widetilde{v}_0, 
\end{align*}
and thus $\overline{v}_0 \leq \widetilde{v}_0$. 
Since $(s(0), c_0, s(1)) \in \delta$, it holds from the definition of the transition relation in \rdef{transition_relation} that $g(\widetilde{v}_0, c_0, w_{\max}) \leq \widetilde{v}_1$. Thus, we obtain 
\begin{equation*}
\overline{v}_1 = g(\overline{v}_0, c_0, w_{\max}) \leq g(\widetilde{v}_0, c_0, w_{\max}) \leq \widetilde{v}_1, 
\end{equation*}
where we have used the monotonicity property of the function $g$ in \req{monotone}. Using the same procedure, we recursively obtain $\overline{v}_k \leq \widetilde{v}_k$, $\forall k\in \mathbb{N}_{0:L}$. Moreover, from conditions (D.2), (D.3) in \rdef{transition_relation2} it holds that $\widetilde{v}_k = \gamma (s(k)) \leq v_{k, \max}$, $\forall k\in \mathbb{N}_{0:L}$, and thus $\overline{v}_k \leq \widetilde{v}_k \leq v_{k, \max}$, $\forall k\in \mathbb{N}_{0:L}$. Also, from $s(L) \in S_{A, final}$ and \req{vfinal}, it holds that 
\begin{equation*}
\widetilde{v}_L \leq \max \{\varepsilon\in\mathbb{R}:  {\cal B}_{\varepsilon} (\hat{x}_L) \subseteq {\cal X}_F\} = v_{final}. 
\end{equation*}
Thus, we obtain $\overline{v}_L \leq \widetilde{v}_L \leq v_{final}$. From condition (D.4), $c_k = 1$ implies that 
\begin{align*}
\alpha_u \circ \underline{\alpha}^{-1} (\overline{v}_k) + \rho_u (\|\hat{u}_k\|) &\leq \alpha_u \circ \underline{\alpha}^{-1} (\widetilde{v}_k) + \rho_u (\|\hat{u}_k\|) \notag \\
&\leq u_{\max}, 
\end{align*}
for all $k\in\mathbb{N}_{0:L-1}$.
As a consequence, the sequence $\overline{v}_0$, $\overline{v}_1$, $\ldots$, $\overline{v}_L$ fulfills conditions (C.1)--(C.4) in \rlem{control_strategy_result}. 
Therefore, from \rlem{control_strategy_result}, by implementing the control strategy in \req{closed} and \req{open} according to $c_k$, $k\in \mathbb{N}_{0:L-1}$, the resulting state trajectory $x_0, \ldots, x_L$ becomes valid for any $x_0 \in {\cal X}_I$ and $w_0, \ldots, w_{L-1} \in {\cal W}$. The proof is complete. 
\end{proof}
\rthm{main_result} states that the existence of an accepting run of ${\cal T}_A$ implies that any state trajectory starting from ${\cal X}_I$ achieves reachability and safety. In order to provide online execution, we generate the communication scheduling in the following way. First, we assume $c_0 = 1$ (i.e., communication occurs at the initial time), which is required since the plant does not know any information about the control inputs to be applied at the initial time $k=0$. Second, in order to reduce the number of communication times as much as possible, we find the accepting run that leads to the smallest number of communication instants. That is, we find $c_k, k\in\mathbb{N}_{0:L-1}$ (with $c_0 = 1$) by solving the following problem: 
\begin{equation}
\min \sum^{L-1} _{k=0} c_k, 
\end{equation}
subject to $(s(0), 0) = s_{A, init}$, $(s(L), L) \in S_{A, final}$, and $((s(k), k), c_k, (s({k+1}), k+1)) \in \delta_A$, $\forall k\in\mathbb{N}_{0:L-1}$. The above optimal run can be obtained as follows. First, for each $s_{A, final} \in S_{A, final}$ we look for an accepting run from $s_{A, init}$ to $s_{A, final}$ by implementing standard graph search methodologies (e.g., Dijkstra algorithm \cite{lavalle}). Then, among all the accepting runs we further select the one with the minimum number of communication instants. 


\begin{algorithm}[t]\label{offline_alg}
{\small 
\SetKwInOut{Input}{input}
\SetKwInOut{Output}{output}
\Input{$\hat{x}_0, \ldots, \hat{x}_L$, $\hat{u}_0, \ldots, \hat{u}_{L-1}$ (reference state and control trajectories),\\ $c_{0}, c_1, \ldots, c_{L-1}$ (communication scheduling), } 
\Output{$x_0 x_1 x_2 \cdots x_L$ (state trajectory)}
    $k = 0$ (initialization); \\
    \While{ $k \leq L$ }{
    \If { $c_k = 0$} { \label{start}
     The plant applies ${u}_k$ (no communication is given) and $k : = k+1$;  
     } 
     \If { $c_k = 1$} {
     The controller receives $x_k$ from the plant; \\ 
     $u_k = \kappa (x_k, \hat{x}_k, \hat{u}_k)$; \label{update_input}\\
     \If {$k < L-1$ and $c_{k+1} = 0$} {
     $\ell^* _k = \mathsf{zeropref} (c_{k+1}, \ldots, c_{L-1})$; \label{zeropref_line} \\ 
     $u_{k+\ell} = \hat{u}_{k+\ell},\ \forall \ell \in \mathbb{N}_{1:\ell^* _k}$; 
     }
     The controller transmits $u_k, \ldots, u_{k+\ell^* _k}$ to the plant; \\
     The plant applies $u_k$ and $k : = k+1$; 
      } \label{finish}
    }
    \caption{Implementation of offline communication scheduling.}
	     }
\end{algorithm}

In summary, the implementation algorithm based on the offline communication scheduling obtained above is illustrated in \ralg{offline_alg}. In the algorithm, the function $\mathsf{zeropref} : \{0, 1\}^{L-k-1} \rightarrow \mathbb{N}$ (\rline{zeropref_line}) is defined by 
\begin{equation}
\mathsf{zeropref} (c_{k+1}, \ldots, c_{L-1}) = {\max}\ \ell,\  {\rm s.t.}\sum^{\ell} _{\ell'=1} c_{k+\ell'} = 0,
\end{equation}
i.e., the function outputs the number of zero-elements appearing at the beginning of the sequence $c_{k+1}, \ldots, c_{L-1}$. As shown in the algorithm, for each communication time the controller {updates} the current control input (\rline{update_input}) by utilizing $x_k$ and a set of control inputs for the non-communication time steps 
(i.e., $u_{k}, \ldots, u_{k+\ell^* _k}$). Then, the control inputs are transmitted to the plant and no communication is given until the next communication time. This procedure is iterated until the terminal step $k=L$ is reached. From \rthm{main_result}, it is shown that any state trajectory starting from $x_0 \in {\cal X}_I$ becomes valid by applying \ralg{offline_alg}. 



\subsection{Discussion on the computational complexity}\label{computation_sec}
Recall that in the naive approach presented in \rsec{error_propagation_sec}, the number of total nodes and edges of a binary tree for the worst case are \textit{exponential} with respect to the time step $L$. Thus, the binary tree construction can be intractable especially for a large value of $L$. 
In the proposed approach, on the other hand, we aim at constructing the timed symbolic error system ${\cal T}_A$, in which the complexity heavily depends on the implementation of \ralg{delta_construct_alg} (i.e., the derivation of $\delta_A$). 
In the algorithm, it is required to check the conditions (D.2)--(D.4) for each $k\in\mathbb{N}_{0:L-1}$ and each transition in ${\cal T}$. Since the total number of transitions in ${\cal T}$ is $2M$ (since there exist two transitions from each $s \in S$), the total number of iterations in \ralg{delta_construct_alg} is $2ML$. 
Since the parameter $M$ is determined independently from $L$,  the construction of ${\cal T}_A$ is \textit{linear} with respect to $L$ and is thus much more tractable than the naive approach. 
Once ${\cal T}_A$ is constructed, it is required to find an offline communication scheduling by finding the accepting run with the minimum number of communication instants. The complexity to find an accepting run from $s_{A,init}$ to each $s_{A,final} \in S_{A, final}$ is $O(ML \ln ML +2ML )$, if we apply the Dijkstra algorithm \cite{lavalle}. Thus, the complexity of finding the optimal communication scheduling is $O(M_f (ML \ln (ML) + 2ML) )$, where $M_f$ denotes the total number of symbols in $S_{A, final}$. 

Note that the total number of iterations in \ralg{delta_construct_alg} as well as complexity to find the communication scheduling also depend on the tuning parameter $M$, which represents the number of partitions of the domain $\mathbb{R}$. Here, if $M$ is selected smaller, we can reduce the complexity to obtain both ${\cal T}_A$ and the offline communication scheduling. However, if $M$ is selected smaller and the partition of $\mathbb{R}$ becomes sparser, the corresponding transition system ${\cal T}_A$ may not approximate precisely enough the error propagation model in \req{vbar}. More specifically, it is possible that the symbol can transition to another one associated with a much larger value than the original behavior in \req{vbar} (i.e., for some $(s_i, c, s_j) \in \delta$, we may have $g (\gamma (s_i), c, w_{\max}) \ll \gamma (s_j)$). 
Due to such mismatch, the offline communication scheduling tends to be more conservative as $M$ is chosen smaller, i.e., the communication frequency tends to be higher as the partition becomes sparser. 
Therefore, the parameter $M$ should be carefully chosen by taking into account the trade-off between the computational complexity and the conservativeness of the communication scheduling.

\subsection{Generating communication plan: an online approach}\label{online_sec}
In the previous subsection we provided an offline framework to generate the communication strategy. This approach is beneficial in terms of the computational load, since communication scheduling does not need to be generated during the online implementation. 
However, the drawback of this approach may be that the communication scheduling is conservative; the generated scheduling may require a higher number of communications than the one that is \textit{actually} (minimally) required to guarantee reachability and safety. This conservativeness is due to the fact that the error between the actual state and the reference during the online implementation may be much smaller than the one assumed by the optimal run generated offline. 
For example, suppose that the accepting run includes the symbol $(s(k), k) \in S_A$ for some $k\in \mathbb{N}_{0:L}$. This implies, from the proof of \rthm{main_result}, that the error between the actual state and the reference is below $\gamma (s(k))$, i.e., $v_k = V(x_k, \hat{x}_k) \leq \gamma (s(k))$. However, since the accepting run is obtained offline, $v_k$ can be much smaller than the upper bound $\gamma (s(k))$. This means that there may exist another $(s' (k), k) \in S_A$ such that $v_k \leq \gamma (s' (k)) < \gamma (s (k))$, i.e., there can exist a symbol that provides a more rigorous upper bound for $v_k$. In this case, there may exist another run from $(s' (k), k)$ providing a smaller number of communication instants than the one from $(s (k), k)$. 

\begin{algorithm}[t]\label{online_alg}
\small{ 
\SetKwInOut{Input}{input}
\SetKwInOut{Output}{output}
\Input{$\hat{x}_0, \ldots, \hat{x}_L$, $\hat{u}_0, \ldots, \hat{u}_{L-1}$ (reference state and control trajectories),}
\Output{$c_{0}, c_1, \ldots, c_{L-1}$ (communication scheduling) \\ 
$x_0 x_1 x_2 \cdots x_L$ (state trajectory)}
$c_0 = 1$, $k = 0$ (initialization); \\
    \For{ $k \leq L$ }{
    \If { $c_k = 0$} { \label{start2}
         The plant applies ${u}_k$ (no communication is given) and $k : = k+1$;  
     } 
     \If { $c_k = 1$} {
     The controller receives $x_k$ from the plant; \\ 
     Let $s (k) = \mathsf{sym} (x_k)$ and \label{com_update_line}
     \begin{equation}\label{optcom_eq}
     c^* _{k+1|k}, \ldots, c^* _{L-1|k} =\mathsf{optcom}(s(k));
     \end{equation}
     
     $u_k = \kappa (x_k, \hat{x}_k, \hat{u}_k)$; \label{update_cur_control_line} \\
     \uIf {$k < L-1$ and $c^* _{k+1|k} = 0$} {
      $\ell^* _k = \mathsf{zeropref} (c^* _{k+1|k}, \ldots, c^* _{L-1|k})$;\\
      $u_{k+\ell} = \hat{u}_{k+\ell},\ \forall \ell \in \mathbb{N}_{1:\ell^* _k}$;  \label{control_noncom_line}\\
      $c_{k+\ell} = 0$, $\forall \ell \in \mathbb{N}_{1:\ell^* _k}$; \\
     \If {$k + \ell^* _k < L-1$} {
      $c_{k+\ell^* _k + 1} = 1$; 
     }
     }
     \Else{$c_{k+1} = 1$;}
     The controller transmits $u_k, \ldots, u_{k+\ell^* _k}$ to the plant; \\
     The plant applies $u_k$ and set $k : = k+1$; 
      } \label{finish2} 
    }
    }
    \caption{Implementation of an online communication scheduling.} 
\end{algorithm}

Motivated by the above, this subsection provides an online communication scheduling algorithm that has the potential to provide a less conservative result than the offline communication case. 
The proposed strategy is illustrated in \ralg{online_alg}. 
In the algorithm, the function $\mathsf{sym}: {\cal X}\rightarrow S$ (\rline{com_update_line}) is defined by 
\begin{equation}\label{identify_current}
\mathsf{sym} (x_k) = \underset{s \in S}{{\rm arg\ min}}\ \gamma (s), \ \ {\rm s.t.}\ V(x_k, \hat{x}_k ) \leq \gamma (s). 
\end{equation}
That is, the function outputs the symbol associated with the \textit{closest} upper bound to $V(x_k, \hat{x}_k)$. Moreover, the function $\mathsf{optcom} : S \rightarrow \{0, 1\}^{L-k-1}$ outputs the optimal communication scheduling by finding an appropriate sequence from $(s(k), k)$ to the symbol in $S_{A,final}$ with the minimum number of communication instants, i.e., 
\begin{equation}\label{opt_com_func}
\mathsf{optcom}(s (k)) = \underset{c_{k+1|k}\cdots c_{L-1|k}}{{\rm arg\ min}} \sum^{L-k-1} _{\ell=1} c _{k+\ell|k}, 
\end{equation}
subject to $((s(k+\ell), k+\ell), c_{k+\ell|k}, (s ({k+\ell+1}), k+\ell+1)) \in \delta_A$, $\forall \ell \in \mathbb{N}_{0:L-k-1}$, with $c_{k|k} = c_k$ and $(s({L}), L) \in S_{A,final}$. 

As shown in the algorithm, for each communication time the controller identifies the current symbol associated with the closest upper bound to $V(x_k, \hat{x}_k)$, aiming at reducing the conservativeness with respect to the offline approach. Then, based on the current symbol it updates the optimal communication scheduling by finding the smallest number of communication instants (\rline{com_update_line}). 
Then, the controller computes a current control input (\rline{update_cur_control_line}) as well as a set of control inputs for the non-communication time steps (\rline{control_noncom_line}) and transmits them to the plant. Note that as shown in the algorithm, the communication is given in a self-triggered manner \cite{heemels2012a}, in which for each communication time the controller determines the next communication time (if it exists) based on the state information $x_k$. 

{
In addition to the communication reduction, another advantage of employing the online approach (rather than the offline approach) is that it can potentially handle larger size of disturbances. This is due to the fact that \ralg{online_alg} looks for an accepting run at $k=0$ for a \textit{given} initial state $x_0 \in {\cal X}_I$, while the offline approach looks for an accepting run, before implementing \ralg{offline_alg}, such that \textit{every} initial state $x_0 \in {\cal X}_I$ should lead to a valid trajectory. Mathematically, we have $\gamma (s(0)) \leq \gamma (s_{init})$ with $s(0) = \mathsf{sym} (x_0)$, meaning that the initial error is over-estimated when the offline approach is employed. 
Thus, an accepting run from $s(0)$ is more likely to be found instead of $s_{init}$ under the same values of $w_{\max}$, or, in other words, a larger size of disturbance is allowed by applying the online approach. 
The above observation will be also illustrated in the simulation section: see \rsec{example_sec} for the linear case. 
}

With a slight abuse of \rdef{accepting_def}, we define the accepting run as follows. 
For given $k \in\mathbb{N}_{0:L-1}$, $x_k \in {\cal X}$, and $c_{k+\ell|k} \in\{0,1\}$, $\forall \ell \in\mathbb{N}_{0:L-k-1}$, the sequence $(s(k), k), (s({k+1}), k+1), \ldots, (s({L}), L)$ is called an \textit{accepting run} for $c_{k+\ell|k}$, $\ell \in\mathbb{N}_{0:L-k-1}$, if $s(k) = \mathsf{sym} (x_k)$,  $((s ({k+\ell}), k+\ell), c_{k+\ell|k}, (s ({k+\ell+1}), k+\ell+1)) \in \delta_A$, $\forall \ell \in \mathbb{N}_{0:L-k-1}$, and $(s ({L}), L) \in S_{A,final}$. 
Moreover, $c_{k+\ell|k} \in\{0,1\}$, $\forall \ell \in\mathbb{N}_{0:L-k-1}$ is called \textit{accepted by ${\cal T}_A$}, if there exists an accepting run for $c_{k+\ell|k}$, $\forall \ell \in\mathbb{N}_{0:L-k-1}$. 
For the online communication approach, we obtain the following result:

\begin{mythm}\label{second_result}
\normalfont
For a given $x_0 \in {\cal X}_I$, suppose that \ralg{online_alg} is implemented. Moreover, suppose that at the initial time $k=0$ there exists $c_{\ell|0}\in \{0, 1\}$, $\ell \in\mathbb{N}_{0:L-1}$ with $c_{0|0} = c_0 = 1$, such that it is accepted by ${\cal T}_A$. 
Then, for any $w_0, \ldots, w_{L-1} \in {\cal W}$, the following holds: 
\renewcommand{\labelenumi}{(E.\arabic{enumi})}
\begin{enumerate}
\item (Feasibility): for any $k\in\mathbb{N}_{1:L-1}$ with $c_k = 1$, there exists $c_{k+\ell|k}\in \{0, 1\}$, $\ell \in\mathbb{N}_{0:L-k-1}$, 
such that it is accepted by ${\cal T}_A$. 

\smallskip
\item (Validity): the resulting state trajectory is valid. \qedwhite
\end{enumerate}
\end{mythm}

The first result (E.1) means that the existence of an accepting run at $k=0$ implies the existence of an accepting run for all the communication time steps afterwards. This property is important, since if \textit{no} accepting runs {were} present for some $k$, the controller would not find a suitable communication scheduling according to \req{optcom_eq}. The second result (E.2) shows the validity of the state trajectory by applying \ralg{online_alg}. 
\begin{proof}
Here we provide a proof only for (E.1). 
The proof for (E.2) is similar to the one of \rthm{main_result} and is provided in Appendix. The proof for (E.1) is given by induction. 
From the assumption, at $k=0$ there exists a communication scheduling such that it is accepted by ${\cal T}_A$. Thus, the controller can find the optimal communication scheduling according to \req{optcom_eq}. Let $c^* _{\ell|0}$, $\ell\in\mathbb{N}_{0:L-1}$ with $c^* _{0|0} = c_0 = 1$ be the optimal communication scheduling obtained at $k=0$ and 
\begin{equation}\label{path_initial}
(s^* ({0}), 0), (s^* ({1}), 1), \ldots, (s^* ({L}), L)
\end{equation}
be the corresponding (accepting) run for $c^* _{\ell|0}$, $\ell\in\mathbb{N}_{0:L-1}$ with $s^* (0) = s (0) = {\mathsf{sym} (x_0)}$. Note that since \req{path_initial} is accepting, we have $(s^* (L), L) \in S_{A, final}$. 
Since $s^* (0)$ is given by solving \req{identify_current}, it holds that $v_0 \leq \gamma(s^* ({0}))$. Thus, we obtain 
\begin{align*}
v_{1} \leq g(v_0, c^* _{0|0}, w_{\max}) \leq g(\gamma(s^* ({0})), c^* _{0|0}, w_{\max}) \leq \gamma(s^* (1)), 
\end{align*}
where $v_1 = V(x_1, \hat{x}_1)$ and the last inequality follows from the fact that $(s^* ({0}), c^* _{0|0}, s^* ({1})) \in \delta$ and $\delta$ is given according to \rdef{transition_relation}. 
Thus, we obtain $v_1\leq \gamma(s^* (1))$. 
By using the same procedure, we recursively obtain 
\begin{equation}\label{vllupper}
v_{k} \leq \gamma(s^* ({k})),\ \forall k \in\mathbb{N}_{0: k_1}, 
\end{equation} 
where $k_1 >0$ denotes the next communication time from the initial time $k=0$ (i.e., $c^* _{1|0} = \cdots = c^* _{k_1-1|0} = 0$, $c^* _{k_1|0} = 1$). 
Now, consider the next communication time $k_1$, and let ${s} ({k_1}) = \mathsf{sym} (x_{k_1})$. In what follows, we show that there exists a communication scheduling for $k_1$, such that it is accepted by ${\cal T}_A$. 
Let $\hat{c}_{k_1+\ell|k_1}$, $\ell \in\mathbb{N}_{0:L-k_1-1}$ be the communication scheduling given by 
\begin{equation}\label{chat_eq}
\hat{c}_{k_1+\ell|k_1} = c^* _{k_1+\ell|0}, \ \forall \ell \in\mathbb{N}_{0:L-k_1-1}. 
\end{equation}
Here, $\hat{c}_{k_1 + \ell|k_1}$, $\ell \in\mathbb{N}_{0:L-k_1}$ represents a \textit{candidate} communication scheduling for $k_1$, such that it is accepted by ${\cal T}_A$, or in other words, there exists an accepting run for $\hat{c}_{k_1 + \ell|k_1}$, $\ell \in\mathbb{N}_{0:L-k_1}$. Let $\hat{s} (k_1 +\ell) \in S$, $\ell \in\mathbb{N}_{0:L-k_1}$ be such that $(\hat{s}(k_1+\ell), \hat{c}_{k_1+\ell|k_1}, \hat{s}(k_1 + \ell + 1)) \in \delta$, $\forall \ell \in \mathbb{N}_{0:L-k_1-1}$ with $\hat{s} (k_1) = s(k_1)$. Note that since ${\cal T}$ is deterministic and nonblocking (see \rsec{abstraction_sec}), the sequence $\hat{s} (k_1 +\ell) \in S$, $\ell \in\mathbb{N}_{0:L-k_1}$ is uniquely determined by $\hat{c}_{k_1 + \ell|k_1}$, $\ell \in\mathbb{N}_{0:L-k_1}$. We now show that 
$(\hat{s}(k_1), k_1), (\hat{s}(k_1+1), k_1+1), \ldots, (\hat{s}(L), L)$ is an accepting run for $\hat{c}_{k_1 + \ell|k_1}$, $\ell \in\mathbb{N}_{0:L-k_1}$. 
Let us first show $((\hat{s}(k_1), k_1), \hat{c}_{k_1|k_1}, (\hat{s}(k_1 + 1), k_1 + 1) \in \delta_A$ (with $\hat{s}(k_1) = s (k_1)$). Condition (D.1) in \rdef{transition_relation2} trivially holds since $(\hat{s}(k_1), \hat{c}_{k_1|k_1}, \hat{s}(k_1 + 1)) \in \delta$. 
Moreover, since $s ({k_1}) = \mathsf{sym}(x(k_1))$ and $v_{k_1} \leq \gamma (s^* ({k_1}))$ (see \req{vllupper}), we obtain 
\begin{align}
\gamma (\hat{s} ({k_1})) &= {\min}\left\{\gamma (s) : s\in S,\ V(x_{k_1},\hat{x}_{k_1}) \leq \gamma (s) \right\} \notag \\
&\leq \gamma (s^* ({k_1})). 
\end{align} 
Since $((s^* ({k_1}), k_1), c^* _{k_1|0}, (s^* ({k_1+1}), k_1 + 1) \in \delta_A$, we obtain $\gamma (\hat{s} ({k_1}))\leq \gamma (s^* ({k_1})) \leq v_{k_1, \max}$ and thus condition (D.2) in \rdef{transition_relation2} holds. 
Moreover, due to the monotonicity property of $g$ and $\hat{c}_{k_1|k_1} = c^* _{k_1|0}$, we obtain 
\begin{equation}\label{eq1}
g(\gamma (\hat{s} ({k_1})), \hat{c}_{k_1|k_1}, w_{\max}) \leq g(\gamma (s^* ({k_1})), c^* _{k_1|0}, w_{\max}).
\end{equation} 
Since $(s^* ({k_1}), c^* _{k_1|0}, s^* ({k_1+1}))$ $\in \delta$, it follows from \req{gcond} in \rdef{transition_relation} that $g(\gamma (s^* ({k_1})), c^* _{k_1|0}, w_{\max}) \leq \gamma (s^* ({k_1} + 1))$, so that we have $g(\gamma (\hat{s} ({k_1})), \hat{c}_{k_1|k_1}, w_{\max}) \leq \gamma (s^* ({k_1} + 1))$. Thus, we obtain
\begin{align}
\gamma (\hat{s} (k_1 + 1))  &=  \min \{\gamma(s): s\in S, \notag \\ 
&\ \ \ \ \ \ \ \ \ \ g (\gamma (\hat{s}(k_1)), \hat{c}_{k_1|k_1}, w_{\max}) \leq \gamma (s)\}  \notag \\
& \leq \gamma (s^* ({k_1} + 1)) \leq {v}_{k_1+1, \max},
\end{align} 
where the last inequality follows from the fact that $((s^* ({k_1}), k_1), c^* _{k_1|0}, (s^* ({k_1+1}), k_1 + 1) \in \delta_A$. Thus, condition (D.3) holds. 
Finally, noting that $\hat{c} _{k_1|k_1} =1$ and
\begin{align}
\alpha_u \circ \underline{\alpha}^{-1} & (\gamma (\hat{s}({k_1}))) + \rho (\hat{u}_{k_1}) \notag \\
&\leq \alpha_u \circ \underline{\alpha}^{-1} (\gamma ({s}^* ({k_1})))+ \rho (\hat{u}_{k_1}) \notag \\
& \leq u_{\max}, 
\end{align} 
it is shown that condition (D.4) in \rdef{transition_relation2} holds. 
Therefore, we have $((\hat{s}(k_1), k_1), \hat{c}_{k_1|k_1}, (\hat{s}(k_1 + 1), k_1 + 1) \in \delta_A$. By using the same procedure as above, we recursively obtain 
\begin{equation}
\gamma (\hat{s} (k_1 + \ell)) \leq \gamma ({s}^* (k_1 +\ell)) \leq {v}_{k_1+\ell, \max},
\end{equation} 
for all $\ell \in \mathbb{N}_{0:L-k_1}$, and it is shown that $((s(k_1+\ell), k_1+\ell), \hat{c}_{k_1+\ell|k_1}, (\hat{s}(k_1 + \ell+1), k_1 + \ell + 1) \in \delta_A$, $\forall \ell \in \mathbb{N}_{0:L-k_1-1}$. 
Since $(s^* ({L}), L) \in S_{A, final}$ and $\gamma (\hat{s} ({L})) \leq \gamma (s^* ({L}))$, we then obtain ${\cal B}_{\gamma(\hat{s} ({L}))} (\hat{x}_L) \subseteq {\cal B}_{\gamma({s}^* ({L}))} (\hat{x}_L) \subseteq {\cal X}_F$. Thus, it holds that $(\hat{s} ({L}), L ) \in S_{A, final}$. Therefore, it is shown that $(\hat{s}(k_1), k_1), (\hat{s}(k_1+1), k_1+1), \ldots, (\hat{s}(L), L)$ is an accepting run for $\hat{c}_{k_1 + \ell|k_1}$, $\ell \in\mathbb{N}_{0:L-k_1}$. 

Since there exists an accepting run for $k_1$, the controller can find the optimal communication scheduling at $k_1$ according to \req{optcom_eq}. Let $c^* _{k_1+\ell|k_1}$, $\forall \ell \in\mathbb{N}_{0:L-k_1-1}$ be the optimal communication scheduling obtained at $k_1$, and let $k_2 > k_1$ be the next communication time from $k_1$ (i.e., $c^* _{k_1+1|k_1} = 0, \ldots, c^* _{k_2-1|k_1} = 0$ and $c^* _{k_2|k_1} = 1$). Let $s(k_2) = \mathsf{sym} (x_{k_2})$ and 
\begin{equation}
\hat{c}_{k_2+\ell|k_2} = c^* _{k_2+\ell|k_1}, \ \forall \ell \in\mathbb{N}_{0:L-k_2-1}
\end{equation}
be a candidate communication scheduling for $k_2$. With a slight abuse of notation, let $\hat{s} (k_2 +\ell) \in S$, $\ell \in\mathbb{N}_{0:L-k_2}$ be such that $(\hat{s}(k_2+\ell), \hat{c}_{k_2+\ell|k_2}, \hat{s}(k_2 + \ell + 1)) \in \delta$, $\forall \ell \in \mathbb{N}_{0:L-k_2-1}$ with $\hat{s} (k_2) = s(k_2)$. By using exactly the same procedure as for the case $k_1$ described above, it follows that 
\begin{equation}\label{pathk2}
(\hat{s} ({k_2}), k_2), (\hat{s} ({k_2 + 1}), k_2 +2), \ldots, (\hat{s} ({L}), L)
\end{equation}
is an accepting run for $\hat{c}_{k_2+\ell|k_2}$, $\ell \in\mathbb{N}_{0:L-k_2-1}$ with $\hat{s} ({k_2}) = s(k_2) = \mathsf{sym} (x_{k_2})$. 
Since there exists an accepting run at $k_2$, the controller can find the optimal communication scheduling according to \req{optcom_eq}. Then, it is again shown that there exists an accepting run at the next communication time $k_3$. Therefore, it is inductively shown that for \textit{any} communication time step $k_1, k_2, k_3, \ldots$, there exists a communication scheduling such that it is accepted by ${\cal T}_A$. The proof of (E.1) is complete. 
\end{proof}

\section{Illustrative examples}\label{example_sec}
In this section we provide two illustrative examples to
validate our control schemes. All simulations were conducted by Matlab 2016a on a Windows 10, Intel(R) Core(TM), 2.40GHz, 8GB RAM computer. 

\subsection{Linear case}
One of the well-known examples of \rpro{problem} is a motion planning problem of a vehicle, in which we aim to steer a vehicle to a desired goal set in finite time while avoiding obstacles. Let $p = [p_x; p_y] \in \mathbb{R}^2$ and $v = [v_x; v_y] \in \mathbb{R}^2$ be the position and velocity of the vehicle, respectively. Defining the state as $x = [p;\ v] \in \mathbb{R}^4$, the dynamics is assumed to be given by 
\begin{equation} \label{sim_system}
\dot{{x}}  =  \left [
\begin{array}{cccc}
0  &  0 & 1 & 0  \\
0 &  0  & 0 & 1  \\
0 & 0 & -\frac{1}{\tau_x}& 0 \\
0 & 0 & 0 & -\frac{1}{\tau_y}
\end{array}
\right ] x + \left [
\begin{array}{cc}
0 & 0 \\
0 & 0 \\
\frac{1}{\tau_x} & 0 \\ 
0 & \frac{1}{\tau_y} \\
\end{array}
\right ] u + w, 
\end{equation}
where $\tau_x = \tau_y = 0.95$, $u \in \mathbb{R}^2$ is the control input and $w \in \mathbb{R}^2$ is the disturbance. We discretize \req{sim_system} under a sample-and-hold controller with sampling time interval $0.5$ to obtain the corresponding discrete-time system: $x_{k+1} = A x_k + B u_k + w_k$. The input and the disturbance sets are assumed to be given by ${\cal U} = \{ u \in  {\mathbb{R}}^{2} :  \|u\| \leq 5\}$, ${\cal W} = \{ w \in  {\mathbb{R}}^{4} :  \|w\| \leq 0.1\}$, and the position of the vehicle $p$ is constrained by the set ${\cal P}$, where the set ${\cal P} \subset \mathbb{R}^2$ is illustrated in \rfig{state_space}. In the figure, the white regions represent the free-space in which the vehicle can move freely, and the black regions represent obstacles to be avoided. Assuming that the velocity $v$ is constrained by the set ${\cal V} = \{ v \in  {\mathbb{R}}^{2} :  \|v\|_\infty \leq 4\}$, the state-space ${\cal X}$ is given by ${\cal X} = {\cal P} \times {\cal V}$. The initial set is given by ${\cal X}_I = {\cal P}_I \times {\cal V}_I$, where ${\cal P}_I = \{p = [p_1; p_2]\in {\cal P} : -9 \leq p_1 \leq -7 \wedge -9 \leq p_2 \leq -7 \} $ and ${\cal V}_I = \{v \in {\cal V} :  v = 0 \}$. The target set is given by ${\cal X}_F = {\cal P}_F \times {\cal V}_F$, where ${\cal P}_F = \{p = [p_1; p_2]\in {\cal P} : 7 \leq p_1 \leq 9 \wedge -9 \leq p_2 \leq -7 \} $ and ${\cal V}_F = \{v \in {\cal V} : \|v\|_\infty \leq 1 \}$. The sets ${\cal P}_I$ and ${\cal P}_F$ are also illustrated in \rfig{state_space}. 

To generate reference state and control trajectories in an offline manner according to \rsec{offline_control_sec}, we have implemented a standard RRT algorithm\cite{lavalle1999}. While implementing the algorithm, we generate a collision-free trajectory with a safety margin $\epsilon = 0.5$, i.e., all states $\hat{x}_0, \ldots, \hat{x}_L$ are at least $\epsilon$ away from the obstacles (i.e., boundaries of ${\cal X}$), see e.g., \cite{karaman2010b}. {Such safety margin is imposed here to obtain large values of $v_{0, \max}, \ldots, v_{L, \max}$, so as to increase the possibility of finding an accepting run of ${\cal T}_A$.}
The algorithm is successfully terminated and finds the reference state and control trajectories $\hat{x}_0, \ldots, \hat{x}_L$, $\hat{u}_0, \ldots, \hat{u}_{L-1}$ with $L= 169$. Based on the trajectories, we compute $v_{k,\max}$, $k \in \mathbb{N}_{0:L}$. Although the naive approach described in \rsec{naive_approach_sec} was implemented, the algorithm did not terminate due to a lack of memory; indeed, the total number of nodes to construct a binary tree for the worst case is $2^{169} \approx 1.7\times 10^{50}$, which is clearly intractable for the algorithm to be terminated. Thus, it is worth applying our proposed approach to alleviate the computational burden. 

\begin{figure}[t]
   \centering
    \subfigure[Illustration of ${\cal P}$.]{
      {\includegraphics[width=4.5cm]{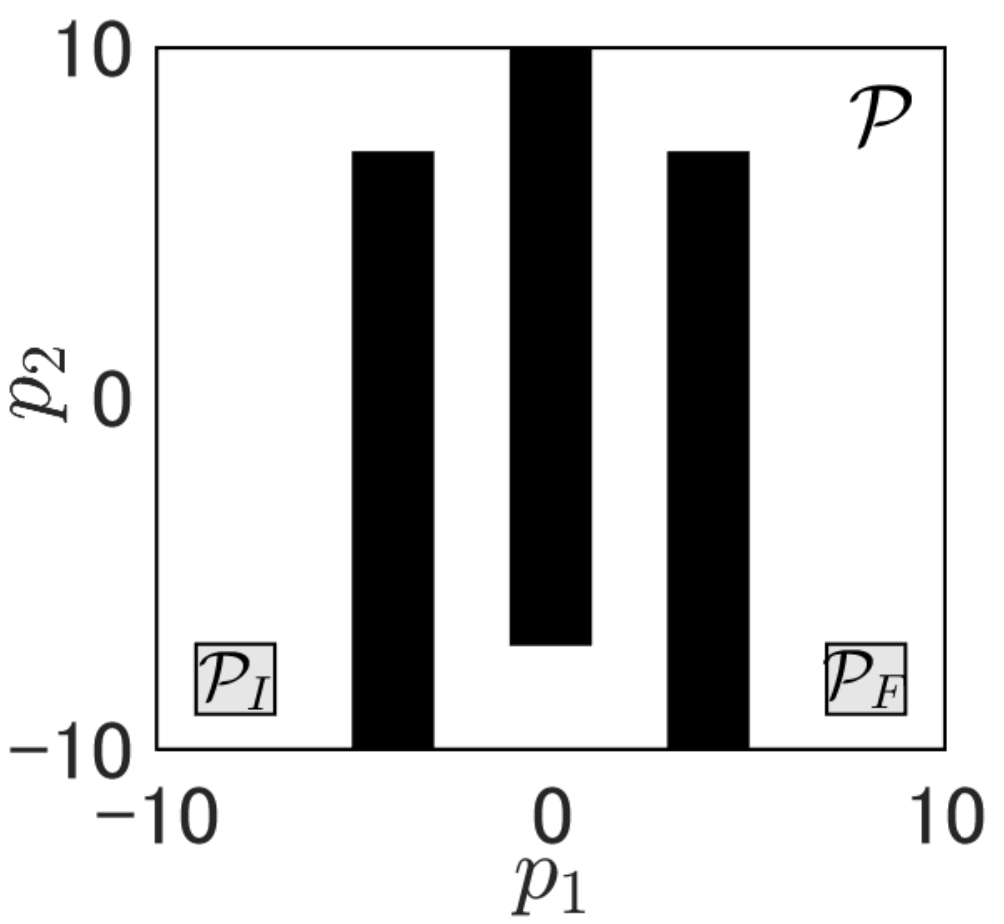}} \label{state_space}}
    \hspace{-0.1cm}
    \subfigure[Trajectory of $p$ (\ralg{offline_alg}).]{
      {\includegraphics[width=4.2cm]{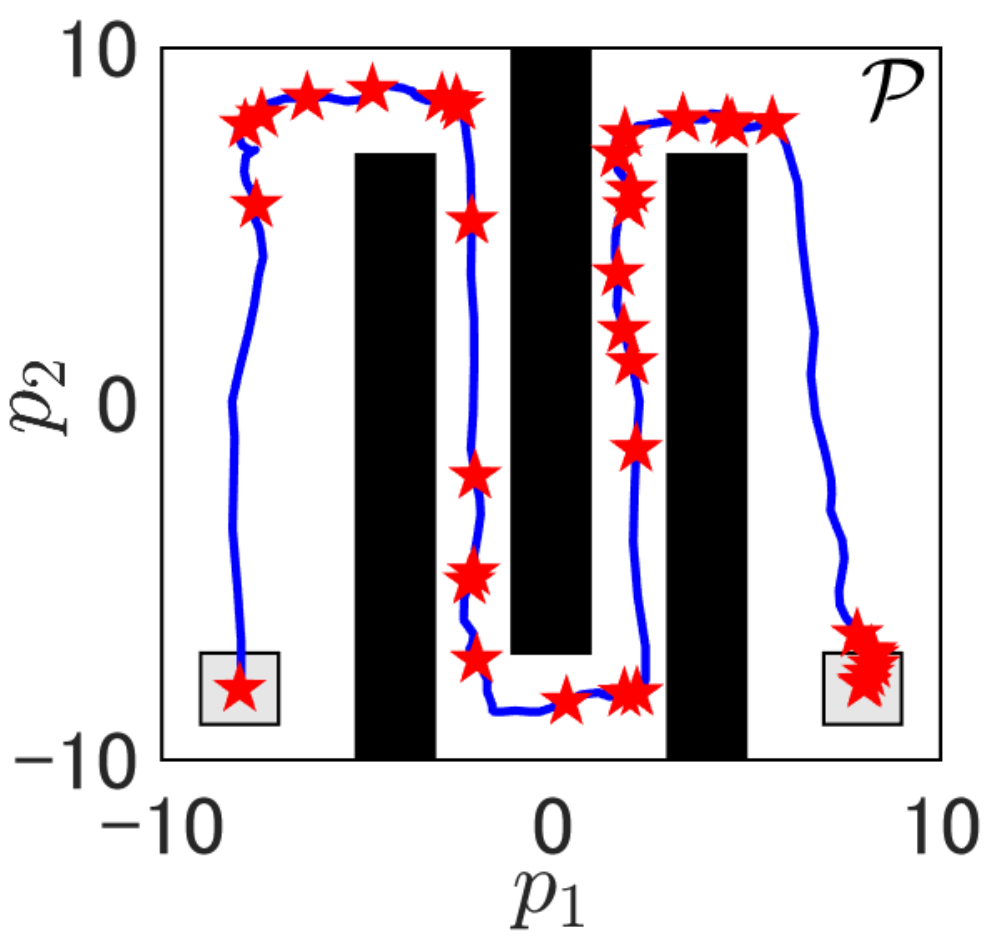}} \label{sample_traj_offline}}
      \hspace{-0.28cm}
      \subfigure[Trajectories of $p$ by applying \ralg{offline_alg} with different initial states.]{
      {\includegraphics[width=4.2cm]{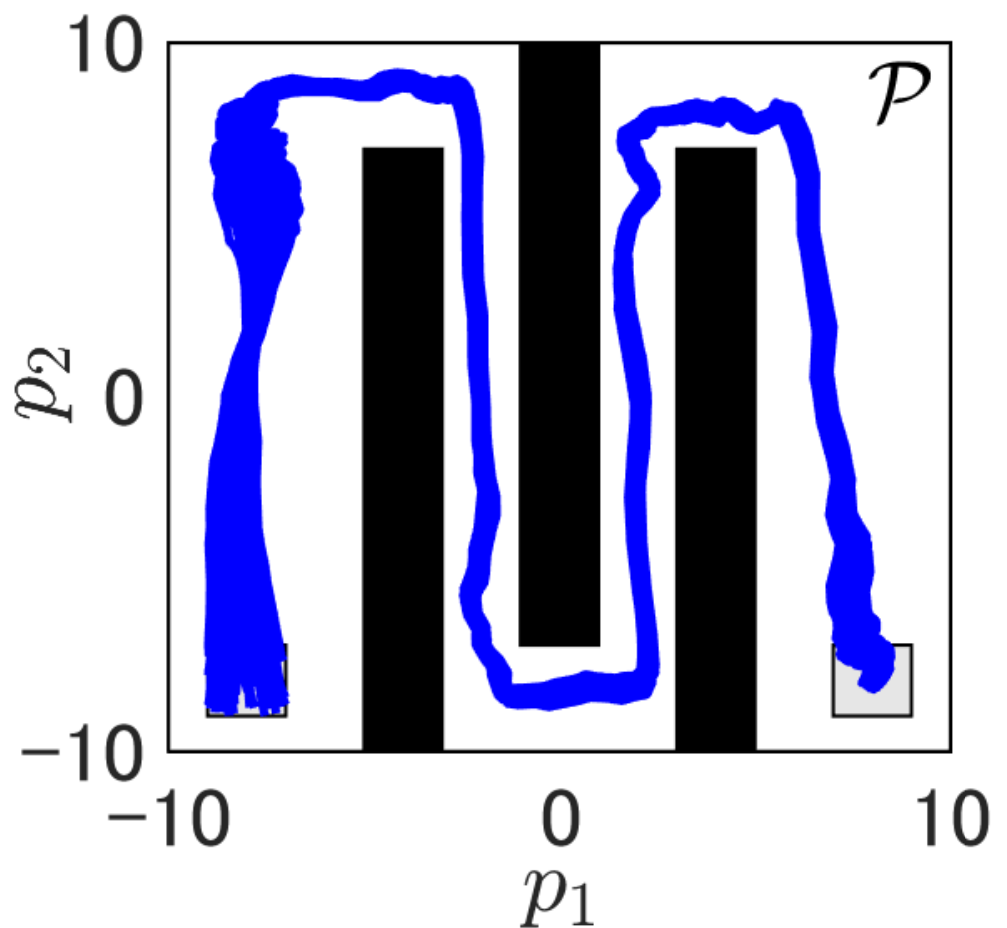}} \label{state_traj_multiple}}
     \hspace{-0.45cm}
    \subfigure[Trajectory of $p$ (\ralg{online_alg}).]{
      {\includegraphics[width=4.4cm]{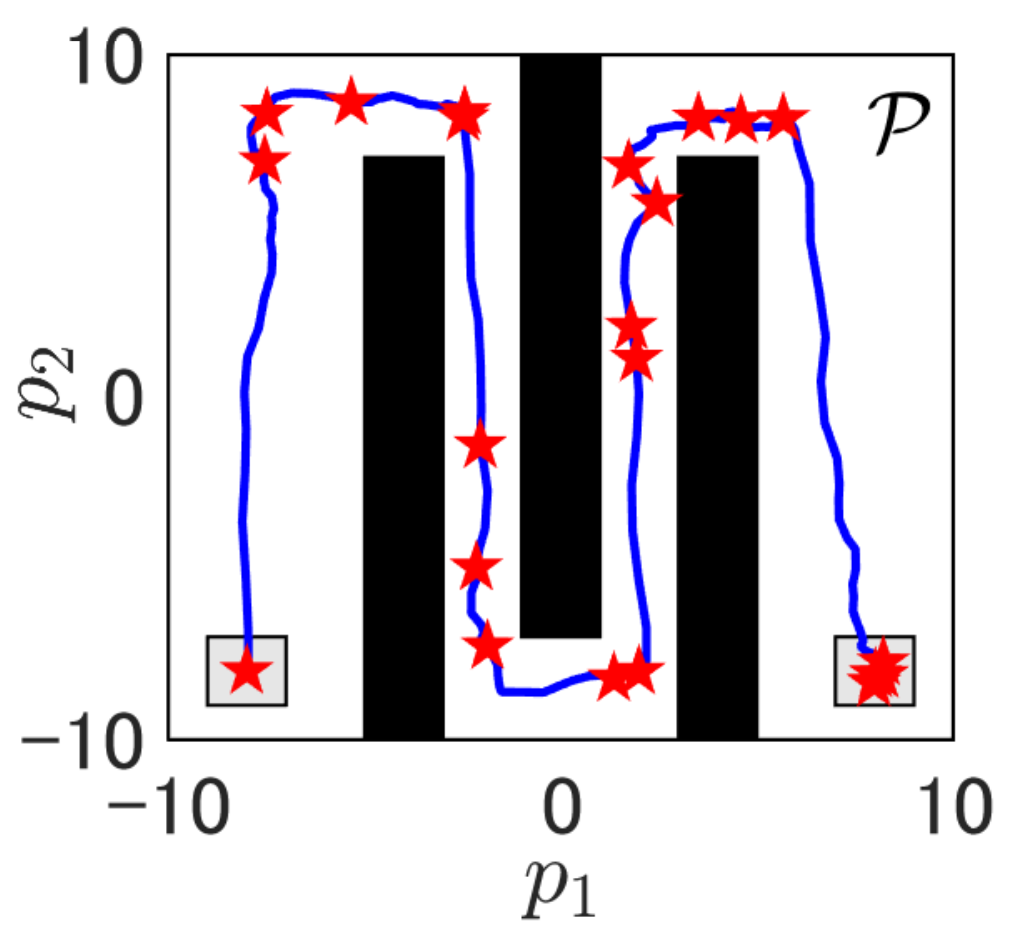}} \label{sample_traj_online}}
      \caption{Illustration of the set ${\cal P}$ and state trajectories of $p$ by applying \ralg{offline_alg} and 3.} 
    \label{state_traj_result}
\end{figure}

\begin{figure}[t]
  \begin{center}
   \includegraphics[width=7.5cm]{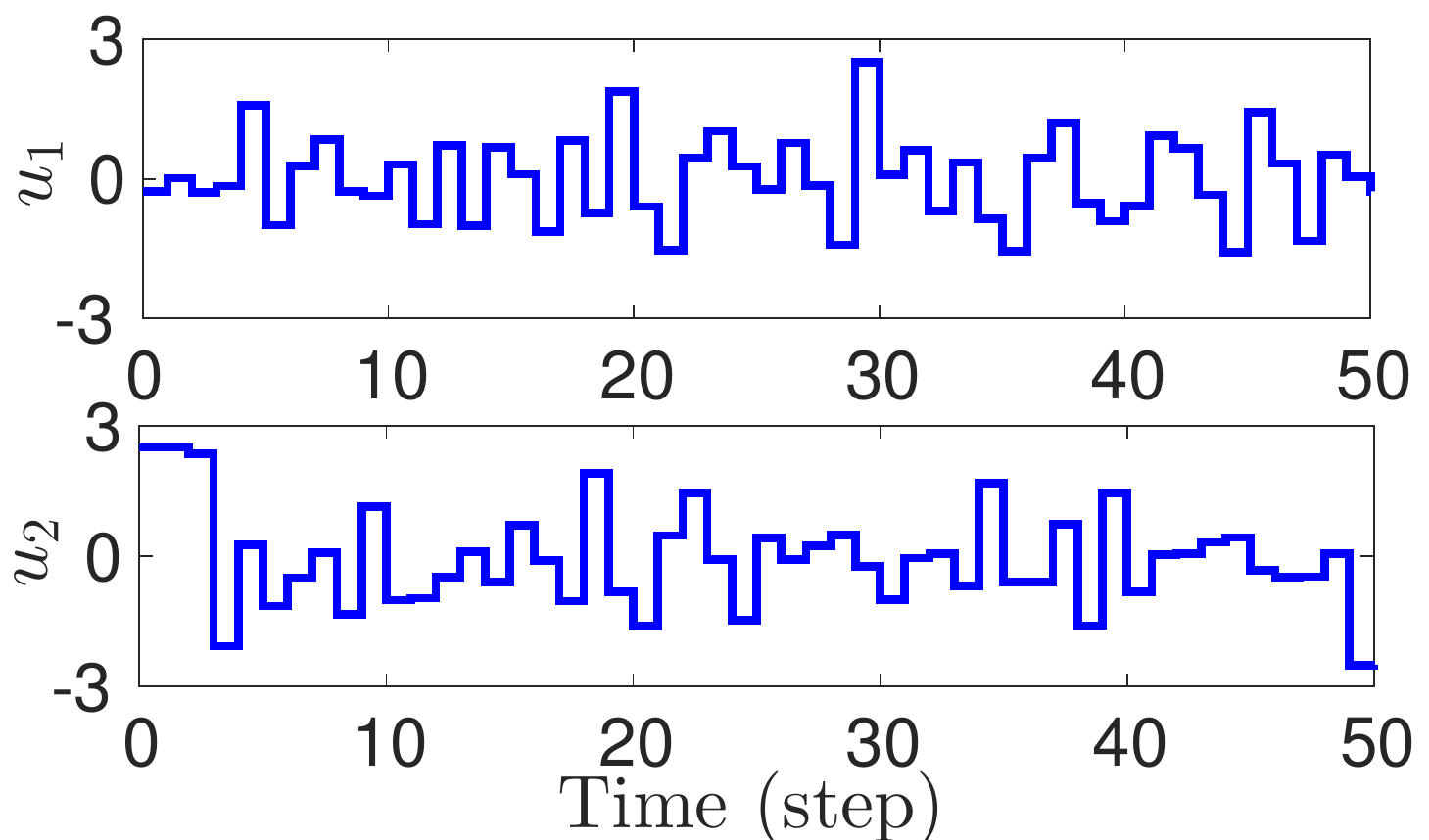}
   \caption{Control inputs over $k\in[0, 50]$ by applying \ralg{offline_alg}. } 
   \label{control_input}
  \end{center}
 \end{figure}

Since it follows that $v_{k,\max} \leq 2$ for all $k \in \mathbb{N}_{0:L-1}$, we set 
$\overline{\nu} = 2$ (see \rrem{select_param}). 
In this subsection, we fix the parameter $M$ as $M=100$, although we will consider different values in the next subsection (nonlinear case). 
Following \rsec{abstraction_sec}, we choose $V(x,y) = \|x-y\|$ as the $\delta$-ISS control Lyapunov function with the control law given by $\kappa (x, y, u) = -K(x-y) + u$, where $K$ is designed so that $\sigma_{\max} (A_{cl}) = 0.7$. Based on this, we have constructed ${\cal T}$ as well as ${\cal T}_A$ by implementing \ralg{delta_construct_alg}. The algorithm took only less than 2 seconds, which is thus shown to be much tractable than the naive approach. 

\rfig{sample_traj_offline} illustrates a sample state trajectory of $p$ by applying the offline communication approach (\ralg{offline_alg}), where $x_0 = [-8; -8; 0; 0]$. 
In the figure, each red star mark represents a state when communication is given (i.e., $x_k$, $k\in\mathbb{N}_{0:L-1}$ with $c_k = 1$). It is shown from the figure that the trajectory enters the target region while avoiding all obstacles. Moreover, the number of communication instants required to achieve reachability is given by $36$ out of the total time steps $L=169$, which is thus shown to achieve the communication reduction. It can be seen from the figure that the transmission frequency tends to be high when the state is close to the obstables, especially when moving in a narrowed space. Intuitively, this is because the safety margins $v_{k,\max}$, $k \in \mathbb{N}_{0:L-1}$ tend to be small especially when the reference trajectory is close to the obstacles, and so the communication is more likely to be given such that the actual state can track to the reference to guarantee safety. \rfig{control_input} illustrates the corresponding control inputs applied to the plant. It is shown from the result that the control inputs satisfy the constraints, i.e., $u_k \in {\cal U}$, $\forall k\in\mathbb{N}_{0:L-1}$. 

To validate \rthm{main_result}, \ralg{offline_alg} has been implemented 100 times with the initial state $x_0$ randomly chosen from ${\cal X}_I$. \rfig{state_traj_multiple} illustrates the resulting state trajectories of $p$. From the figure, all trajectories are indeed shown to achieve reachability while avoiding obstacles, regardless of disturbance influence and initial states. 

\rfig{sample_traj_online} illustrates a sample state trajectory of $p$ by applying the online communication approach (\ralg{online_alg}). Moreover, \rtab{result_tab_linear} illustrates the number of communication instants by applying \ralg{offline_alg} and \ralg{online_alg}, as well as the average computation time required for each communication time instant during the online execution (i.e., the average computation time to execute from \rline{start} to \rline{finish} for \ralg{offline_alg} and from \rline{start2} to \rline{finish2} for \ralg{online_alg}). From \rfig{sample_traj_online} and the table, the online approach is shown to achieve a smaller number of communication instants than the offline approach. On the other hand, it is shown that the computation time for the online approach is longer than the offline approach. This is due to the fact that the online approach is required to update the communication scheduling for each communication time step. 

{
To analyze how much the disturbance size can be tolerated for both the online and offline approaches, we further computed the maximum allowable $w_{\max} > 0$, such that the offline communication scheduling (an accepting run from $s_{init}$) is found, as well as that the online communication scheduling at $k=0$ (an accepting run from $s(0) = \mathsf{sym} (x_0)$ with $x_0 = [-8; -8; 0; 0]$) is found. The results are given by $w_{\max} = 0.12$ for the offline approach, and $w_{\max} = 0.21$ for the online approach. Thus, the results show that a larger size of disturbance is allowed by applying the online approach. As described in \rsec{online_sec}, this is due to the fact that the initial state error is over-estimated when the offline approach is employed (i.e., $\gamma (s(0)) \leq \gamma (s_{init})$). 
}


\begin{table}[t]
\begin{center} 
\caption{The number of communication instants and the average compuataion time during online execution.} \label{result_tab_linear}
{\small 
\begin{tabular}{ccc} \hline 
                          & \ralg{offline_alg} & \ralg{online_alg} \\ \hline \hline
Num. of communication & 36  &  21  \\ \hline
Computation time (s) &  $5\times10^{-4}$ & 0.25  \\ 
\hline \hline
\end{tabular}
}
\end{center}
\end{table}

\subsection{Nonlinear case}\label{sim_nonlinear_sec}
As an example of a nonlinear system, we consider a control problem of an inverted pendulum, whose continuous-time model is discretized with the sampling time interval $\Delta = 0.2$: 
\begin{align*}
{x}_{1,k+1} &=  x_{1,k} + \Delta (x _{2,k} + w_k) \\
{x}_{2, k+1} & = x_{2,k} + \Delta (a \sin x_{1,k} - b x_{2,k} + u_k),  
\end{align*}
where $x_{1, k}$ and $x_{2, k}$ are the states representing the angular position and the velocity of the mass, $u_k \in \mathbb{R}$ is the control input, $w_k \in \mathbb{R}$ is the additive disturbance, and $a, b>0$ are the parameters characterized by the physical quantities such as the gravity constant. 
Assume that $a = 0.6$, $b = 3$ and the control and the disturbance sets are given by ${\cal U} = \{u \in \mathbb{R}: |u| \leq 2 \}$, ${\cal W} = \{w \in \mathbb{R}: |w| \leq 0.01 \}$. Regarding the state constraint, we consider the following two sets: 
\begin{align}
{\cal X}_{A} &= \{x \in \mathbb{R}^2: x = [x_1; x_2] \in \mathbb{R}^2: |x_1| \leq 1 \wedge |x_2| \leq 0.5 \}, \notag\\
{\cal X}_{B} &= \{x \in \mathbb{R}^2: x = [x_1; x_2] \in \mathbb{R}^2: 2|x_1| + 4|x_2| \geq 1 \}, \notag
\end{align}
and ${\cal X} = {\cal X}_A \cap {\cal X}_B$. 
The set ${\cal X}_B$ is given so as to steer the pendulum with sufficiently large position and velocity, see e.g., \cite{saferobotcontrol2016a} for a similar constraint. 
The set ${\cal X}$ is illustrated as the white regions in \rfig{state_space_nonlinear}. 

Let us obtain the upper bound of the error model in \req{vbar} in order to implement the proposed strategies. For $x = [x_1; x_2]\in {\cal X}$ and $y = [y_1; y_2]\in {\cal X}$, it follows that 
\begin{align*}
x^+ _{u, w_1} - y^+ _{u, w_2}
               = & \left[
\begin{array}{cc}
1 & \Delta   \\
a \Delta \eta(x_1, y_1) & 1 - b\Delta 
\end{array}
\right] (x-y) \\ 
& + \left[
\begin{array}{c}
\Delta \\
0
\end{array}
\right] (w_1 - w_2), 
\end{align*}
where $x^+ _{u, w_1} = f(x, u, w_1)$, $y^+ _{u, w_2} = f(y, u, w_2)$, and $\eta (x_2, y_2) = (\sin x_2 - \sin y_2) / (x_2 - y_2)$. Since we have 
\begin{align}
\underset{-1\leq x_1, y_1 \leq 1}{ \min }\ (\sin x_1 - \sin y_1) / (x_1 -y_1) = 0.84, 
\end{align}
the system is Lipschitz continuous satisfying the property in \req{lipschitz} with $L_{x} = 1.03$ and $L_{w} = 0.20$ (see, e.g., \cite{girard2009a} for a related analysis for the continuous case). Moreover, let $V : \mathbb{R}^2 \times \mathbb{R}^2 \rightarrow \mathbb{R}$ be given by $V(x, y) = (x-y)^\mathsf{T} P (x-y)$, where $P = [2.1, 0.45; 0.45, 0.43]$ and $\kappa (x, y, u) = u - k_u (x-y)$ with $k_u = [2.9,\  2.0]$. Then, it can verified that 
\begin{align*}
V ( x^+ _{\kappa, w_1}, &y^+ _{u, w_2})- V (x, y) \\ 
                               & \leq - (x-y)^\mathsf{T} Q (x-y)  + \rho (|w_1-w_2|), 
\end{align*}
where $Q= [0.29, 0; 0, 0.29]$ and $\rho (|w_1-w_2|) = 3.5 | w_1-w_2| + 0.16 |w_1-w_2|^2$. Therefore, it is shown that the function $V$ is $\delta$-ISS control Lyapunov function with respect to $\kappa$, and we can obtain the corresponding error model in \req{vbar}. 

\begin{figure}[t]
   \centering
    \subfigure[State-space ${\cal X}$ (the white region).]{
      {\includegraphics[width=7cm]{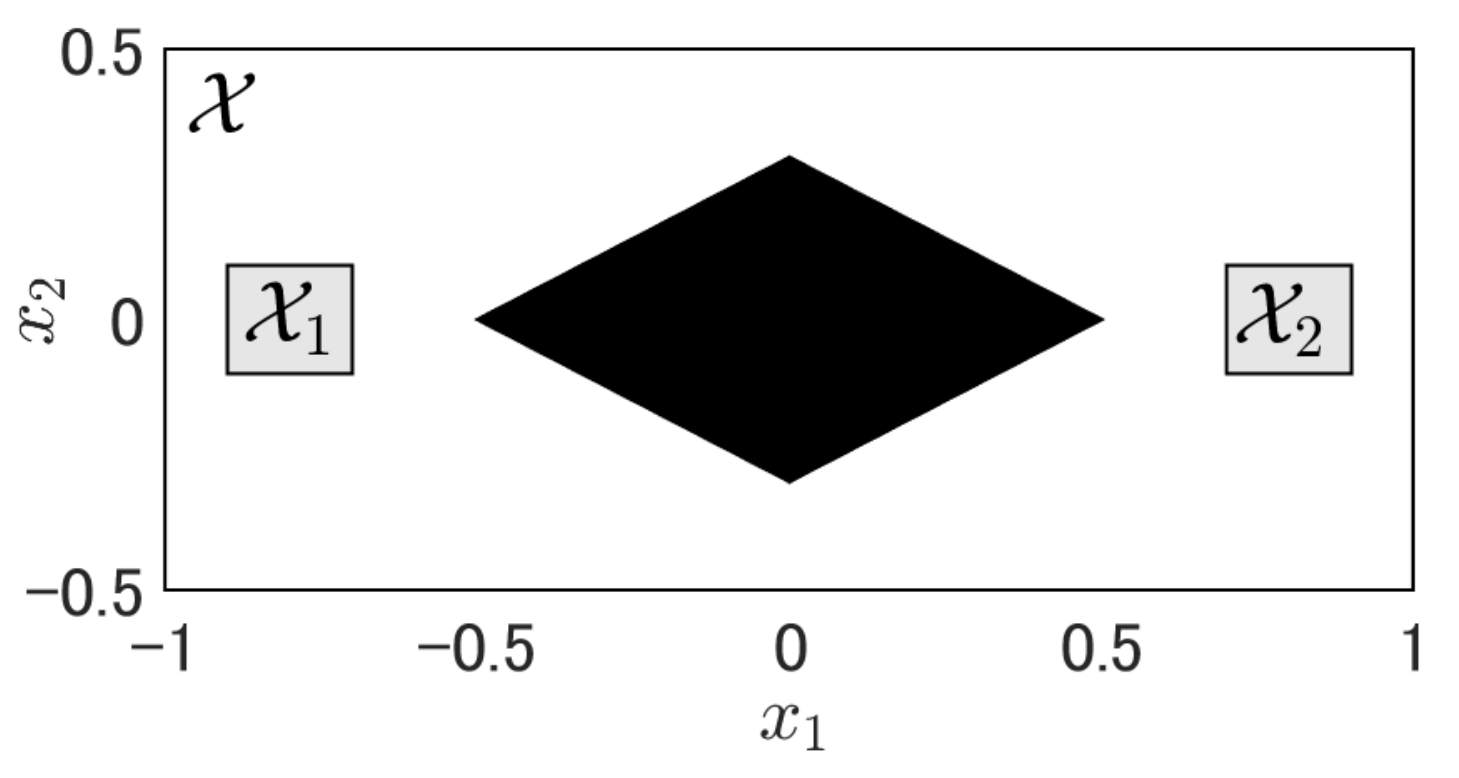}} \label{state_space_nonlinear}}
    \subfigure[State trajectories by applying \ralg{offline_alg} for the time interval $k\in{[0, 1000]}$.]{
      {\includegraphics[width=7cm]{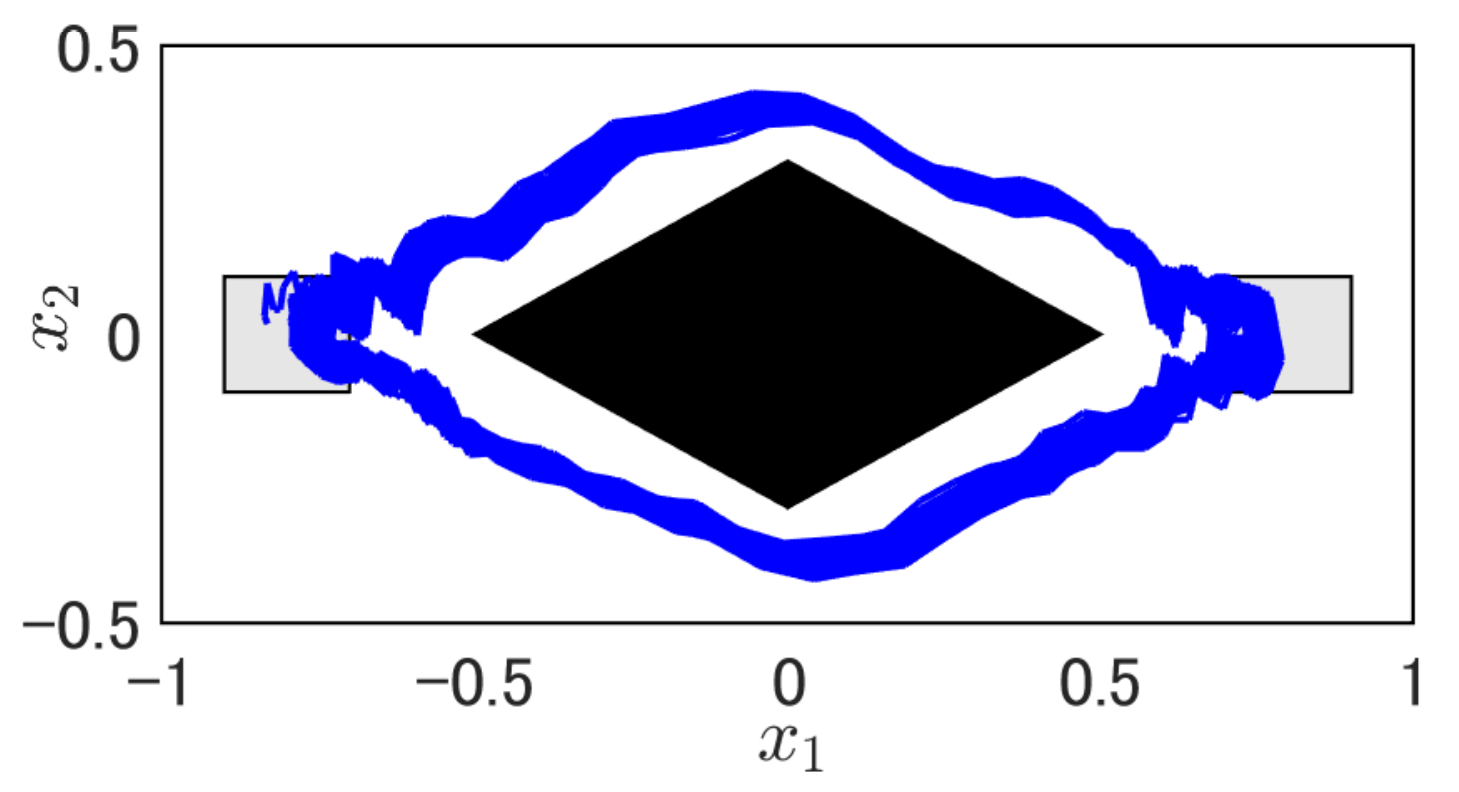}} \label{sample_traj}}
    \label{state_traj_result}
    \caption{State-space ${\cal X}$ and state trajectory by applying \ralg{offline_alg}. } 
\end{figure}

\begin{figure}[t]
  \begin{center}
   \includegraphics[width=8cm]{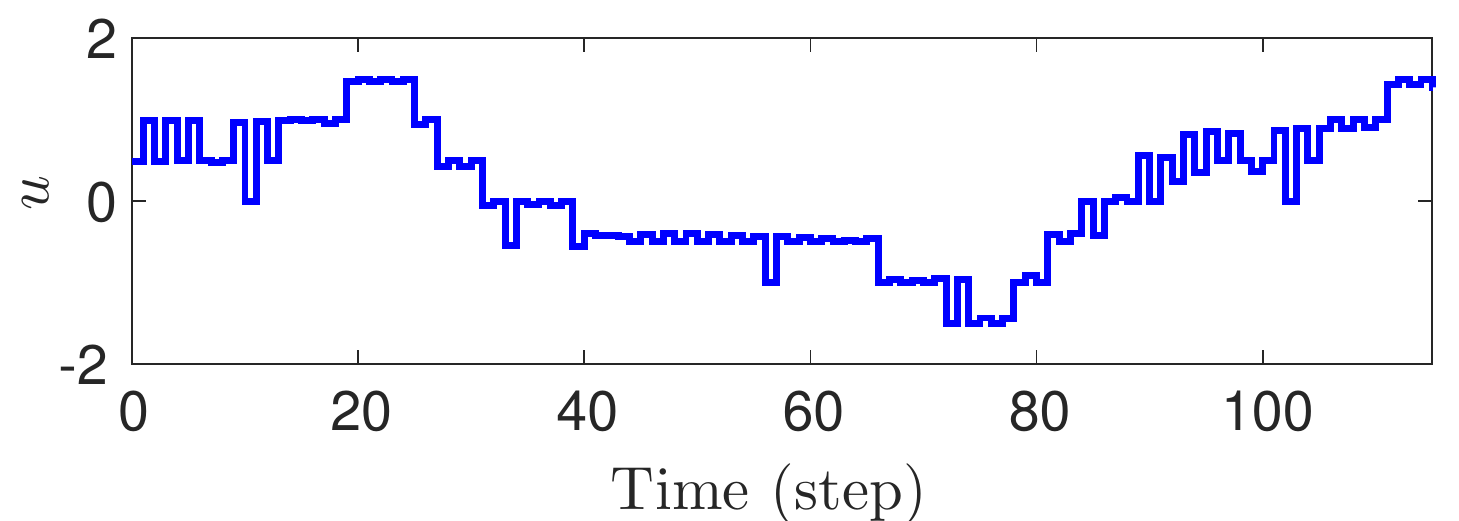}
   \caption{Control inputs over $k\in[0, 120]$ by applying \ralg{offline_alg}. } 
   \label{control_input_nonlinear}
  \end{center}
 \end{figure}

Let ${\cal X}_1 = \{x = [x_1; x_2] \in \mathbb{R}^2: -0.9 \leq x_1 \leq -0.7 \wedge |x_2| \leq 0.1 \}$ and ${\cal X}_2 = \{x = [x_1; x_2] \in \mathbb{R}^2: 0.7 \leq x_1 \leq 0.9 \wedge |x_2| \leq 0.1 \}$. We assume $x_0 = [-0.83; 0.05] \in {\cal X}_1$. Instead of achieving stabilization of the origin, we aim here at designing control and communication strategies such that the state can \textit{periodically traverse} between ${\cal X}_1$ and ${\cal X}_2$. To this aim, we construct ${\cal T}_A$ for both ${\cal X}_I = {\cal X}_1$, ${\cal X}_F = {\cal X}_2$ and ${\cal X}_I = {\cal X}_2$, ${\cal X}_F = {\cal X}_1$, and switch the designed strategies between the two sets during the online implementation. That is, starting from $x_0 \in {\cal X}_1$ we first implement control and communication strategies by setting ${\cal X}_I = {\cal X}_1$, ${\cal X}_F = {\cal X}_2$. Then, once the state enters ${\cal X}_2$ we set ${\cal X}_I = {\cal X}_2$, ${\cal X}_F = {\cal X}_1$ and implement the corresponding strategies. This procedure is iterated for all times so that the state trajectory traverses between ${\cal X}_1$ and ${\cal X}_2$. 

As with the linear case, we implement the RRT algorithm to generate the reference trajectories for both ${\cal X}_I = {\cal X}_1$, ${\cal X}_F = {\cal X}_2$ and ${\cal X}_I = {\cal X}_2$, ${\cal X}_F = {\cal X}_1$. 
For both cases, we set $\nu_{i} = 0.5 i/(M-1)$, $\forall i\in\mathbb{N}_{1:M-1}$ and $M=200$ and construct ${\cal T}_A$ by implementing \ralg{delta_construct_alg}. \rfig{sample_traj} illustrates the resulting state trajectories by implementing the offline communication strategy (\ralg{offline_alg}). The figure shows that the state trajectory actually traverses between ${\cal X}_1$ and ${\cal X}_2$ while remaining inside ${\cal X}$. Moreover, the number of communication instants for the time interval $k \in [0, 1000]$ is $484$, which is thus shown to achieve the communication reduction. 
\rfig{control_input_nonlinear} illustrates the corresponding control inputs applied to the plant. From the result, it is shown that the control inputs satisfy the constraints, i.e., $u_k \in {\cal U}$, $\forall k\in\mathbb{N}_{0:L-1}$. 

In \rsec{computation_sec}, we have discussed the trade-off between the computation time to obtain ${\cal T}_A$ and the number of communication instants according to the selection of $M$. To analyze such trade-off, we construct ${\cal T}_A$ for both ${\cal X}_I = {\cal X}_1$, ${\cal X}_F = {\cal X}_2$ and ${\cal X}_I = {\cal X}_2$, ${\cal X}_F = {\cal X}_1$ with different selection of the parameter $M$ as $M = 400, 100, 50, 10$, and generate the corresponding offline communication schdulings. For each selection of $M$, we measure the {total} computation time to terminate \ralg{delta_construct_alg}. Then, we implement the control and communication strategies such that the state trajectory traverses between ${\cal X}_1$ and ${\cal X}_2$ according to the procedure described above, and count the number of communication instants during the time interval $k \in [0, 1000]$. The results are shown in \rtab{result_tab}. The table illustrates that the number of communication instants increases as $M$ is selected smaller, and the feasible communication scheduling was not found when $M=10$ (with the symbol ``---''); as described in \rsec{computation_sec}, this is because the mismatch between the symbolic model ${\cal T}$ and the original error model in \req{vbar} becomes larger as the partition of the domain $\mathbb{R}$ becomes sparser. 
On the other hand, the computation time to generate the communication scheduling decreases as $M$ is selected smaller, which is due to a decrease of the number of iterations in \ralg{delta_construct_alg}. Therefore, it is shown that there exists a trade-off between the computation time to generate the communication scheduling and the communication load, and the trade-off can be regulated by the tuning parameter $M$. 

\begin{table}[tbp]
\begin{center}
\caption{Number of communication instants during $k \in [0, 1000]$ and the compuataion time to generate the offline communication scheduling.} \label{result_tab}
{\small
\begin{tabular}{ccccc} \hline 
 $M$ & $400$ & $100$ & $50$ & $10$ \\ \hline \hline
Num. of communication & 430  &  484 & 530 & ---  \\ \hline
Computation time (s) &  32 & 5.4 & 0.62 & --- \\ 
\hline \hline
\end{tabular}
}
\end{center}
\end{table}


\section{Conclusion and Future work}
In this paper, we propose control and communication strategies for reachability and safety specifications in a networked control system. The key idea of the proposed approach is to utilize the notion of $\delta$-ISS control Lyapunov function, which captures contractive behaviors between any pair of the state trajectories under a certain state feedback control law. 
The function is given to introduce the error propagation model, which represents how the upper bound of the error between the actual and the reference trajectories behaves according to the occurrence or non-occurrence of communication. Based on the error propagation model, we derive a sufficient condition for the communication scheduling to guarantee reachability and safety. 
Moreover, in order to reduce computational complexity, we introduce the notion of symbolic error system, which represents an abstracted behavior of the error propagation model. The communication scheduling is then given by implementing standard graph search methodologies, and is provided in both offline and online fashion. Finally, we illustrate the benefits of the proposed approach through numerical simulations for both linear and nonlinear cases. 

It should be noted that, incremental stability analysis has been provided recently for complex dynamical systems that have not been considered in this paper, including hybrid systems\cite{postoyan2015}, switched systems\cite{girard2010a}, and mechanical systems\cite{contraction2000}. Thus, future work will involve investigating the applicability of the proposed approach to those types of systems. 
{Moreover, our future work involves investigating the applicability of the proposed approach to \textit{mobile communication networks}, where the network consists of multiple nodes communicating with each other under the randomness of connectivity. The problem may be treated by incorporating the idea of energy-aware packet forwarding protocols, such as those presented in \cite{energyawarerouting1,energyawarerouting2}. Finally, extending the proposed approach to more complex specifications, such as those expressed by Linear Temporal Logic (LTL) formulas, will be taken into account in future investigations. }




\bibliographystyle{IEEEtran}
\bibliography{IEEEabrv,myrefs}

\begin{thebibliography}{10}

\bibitem{gupta2010}
R.~A. Gupta and M.-Y. Chow, ``{Networked Control System: Overview and Research
  Trends},'' \emph{IEEE Transactions on Industrial Electronics}, vol.~57,
  no.~7, pp. 2527--2535, 2010.

\bibitem{wei2001}
W.~Zhang, M.~S. Branicky, and S.~M. Phillips, ``Stability of networked control
  systems,'' \emph{IEEE Control Systems}, vol.~21, no.~1, pp. 84--99, 2001.

\bibitem{heemels2012a}
W.~P. M.~H. Heemels, K.~H. Johansson, and P.~Tabuada, ``An introduction to
  event-triggered and self-triggered control,'' in \emph{Proceedings of the
  51st IEEE Conference on Decision and Control (IEEE CDC)}, 2012, pp.
  3270--3285.

\bibitem{tabuada2007a}
P.~Tabuada, ``Event-triggered real-time scheduling of stabilizing control
  tasks,'' \emph{IEEE Transactions on Automatic Control}, vol.~52, pp.
  1680--1685, 2007.

\bibitem{heemels2011a}
M.~C.~F. Donkers and W.~P. M.~H. Heemels, ``Output-based event-triggered
  control with guaranteed {${\mathcal L}_{\infty}$} gain and decentralized
  event-triggering,'' \emph{IEEE Transactions on Automatic Control}, vol.~57,
  no.~6, pp. 1362--1376, 2011.

\bibitem{dolk2017c}
V.~S. Dolk, D.~P. Borgers, and W.~P. M.~H. Heemels, ``Output-based and
  decentralized dynamic event-triggered control with guaranteed {${\mathcal
  L}_p$}-gain performance and zeno-freeness,'' \emph{IEEE Transactions on
  Automatic Control}, vol.~62, no.~1, pp. 34--49, 2016.

\bibitem{hashimoto2018b}
K.~Hashimoto, S.~Adachi, and D.~V. Dimarogonas, ``Aperiodic sampled-data
  control via explicit transmission mapping: a set-invariance approach,''
  \emph{IEEE Transactions on Automatic Control}, vol.~63, no.~10, pp.
  3523--3530, 2018.

\bibitem{kishida2018}
M.~Kishida, ``Event-triggered control with self-triggered sampling for
  discrete-time uncertain systems,'' \emph{IEEE Transactions on Automatic
  Control}, 2018 (to appear).

\bibitem{romain2014a}
R.~Postoyan, P.~Tabuada, D.~Nesic, and A.~Anta, ``A framework for the
  event-triggered stabilization of nonlinear systems,'' \emph{IEEE Transactions
  on Automatic Control}, vol.~60, no.~4, pp. 982--996, 2014.

\bibitem{tabuada2010a}
A.~Anta and P.~Tabuada, ``To sample or not to sample: Self-triggered control
  for nonlinear systems,'' \emph{IEEE Transactions on Automatic Control},
  vol.~55, no.~9, pp. 2030--2042, 2010.

\bibitem{dimos2012a}
D.~V. Dimagoronas, E.~Frazzoli, and K.~H. Johansson, ``Distributed
  event-triggered control for multi-agent systems,'' \emph{IEEE Transactions on
  Automatic Control}, vol.~57, no.~5, pp. 1291--1297, 2012.

\bibitem{heemels2013a}
W.~P. M.~H. Heemels and M.~C.~F. Donkers, ``Model-based periodic
  event-triggered control for linear systems,'' \emph{Automatica}, vol.~49,
  no.~3, pp. 698--711, 2013.

\bibitem{potoyan2013}
R.~Postoyan, A.~Anta, W.~P. M.~H. Heemels, P.~Tabuada, and D.~Nesic, ``Periodic
  event-triggered control for nonlinear systems,'' in \emph{Proceedings of the
  52nd IEEE Conference on Decision and Control (IEEE CDC)}, 2013, pp.
  7397--7402.

\bibitem{girard2014a}
A.~Girard, ``Dynamic triggering mechanisms for event-triggered control,''
  \emph{IEEE Transactions on Automatic Control}, vol.~60, no.~7, pp.
  1992--1997, 2014.

\bibitem{araujo2013}
J.~Ara{\'u}jo, M.~Mazo, A.~Anta, P.~Tabuada, and K.~H. Johansson, ``{System
  Architectures, Protocols and Algorithms for Aperiodic Wireless Control
  Systems},'' \emph{IEEE Transactions on Industrial Informatics}, vol.~10,
  no.~1, pp. 175--184, 2013.

\bibitem{chen2016a}
C.~Peng, D.~Yue, and M.-R. Fei, ``{A Higher Energy-Efficient Sampling Scheme
  for Networked Control Systems over IEEE 802.14.4 Wireless Networks},''
  \emph{IEEE Transactions on Industrial Informatics}, vol.~12, no.~5, pp.
  1766--1744, 2016.

\bibitem{eventsurvey}
Q.~Liu, Z.~Wang, X.~He, and D.~Zhou, ``A survey of event-based strategies on
  control and estimation,'' \emph{Systems Science \& Control Engineering},
  vol.~2, no.~1, pp. 90--97, 2014.

\bibitem{aircraftsafety}
A.~M. Bayen, I.~M. Mitchell, M.~Oishi, and C.~J. Tomlin, ``Aircraft autolander
  safety analysis through optimal control-based reach set computation,''
  \emph{Journal of Guidance, Control, and Dynamics}, vol.~30, no.~1, pp.
  68--77, 2017.

\bibitem{gerogios2009a}
G.~E. Fainekos, A.~Girard, H.~Kress-Gazit, and G.~J. Pappas, ``Temporal logic
  motion planning for dynamic robots,'' \emph{Automatica}, vol.~45, no.~2, pp.
  343--352, 2009.

\bibitem{saferobotcontrol2016a}
M.~Vukosavljev, I.~Jansen, M.~E. Broucke, and A.~P. Schoellig, ``Safe and
  robust robot maneuvers based on reach control,'' in \emph{IEEE Conference on
  Robotics and Automation (ICRA)}, 2016.

\bibitem{habets2004a}
L.~C. G. J.~M. Habets and J.~H. van Schuppen, ``A control problem for affine
  dynamical systems on a full-dimensional polytope,'' \emph{Automatica},
  vol.~40, no.~1, pp. 21--35, 2004.

\bibitem{belta2004a}
C.~Belta, V.~Isler, and G.~J. Pappas, ``Discrete abstractions for robot motion
  planning and control in polygonal environments,'' \emph{IEEE Transactions on
  Robotics}, vol.~21, no.~5, pp. 864--874, 2004.

\bibitem{temporalmpc2012a}
T.~Wongpiromsarn, U.~Topcu, and R.~M. Murray, ``Receding horizon temporal logic
  planning,'' \emph{IEEE Transactions on Automatic Control}, vol.~57, no.~11,
  pp. 2817--2830, 2012.

\bibitem{girard2012a}
A.~Girard, ``Controller synthesis for safety and reachability via approximate
  bisimulation,'' \emph{Automatica}, vol.~48, no.~5, pp. 947--953, 2012.

\bibitem{girard2009a}
G.~Pola, A.~Girard, and P.~Tabuada, ``Approximately bisimilar symbolic models
  for nonlinear control systems,'' \emph{Automatica}, vol.~44, no.~10, pp.
  2508--2516, 2008.

\bibitem{hashimoto2017d}
K.~Hashimoto, S.~Adachi, and D.~V. Dimarogonas, ``Self-triggered control for
  constrained systems: a contractive set-based approach,'' in \emph{Proceedings
  of 2017 American Control Conference}, 2017, pp. 1011--1016.

\bibitem{incremental}
D.~Angeli, ``A lyapunov approach to incremental stability properties,''
  \emph{IEEE Transactions on Automatic Control}, vol.~47, no.~3, pp. 410--421,
  2002.

\bibitem{incremental_discrete}
D.~N. Tran, B.~S. Ruffer, and C.~M. Kellett, ``Incremental stability properties
  for discrete-time systems,'' in \emph{Proceedings of the 55th IEEE Conference
  on Decision and Control}, 2016, pp. 477--482.

\bibitem{girard2010a}
A.~Girard, G.~Pola, and P.~Tabuada, ``Approximately bisimilar symbolic models
  for incrementally stable switched systems,'' \emph{IEEE Transactions on
  Automatic Control}, vol.~55, no.~1, pp. 116--126, 2010.

\bibitem{lavalle1999}
S.~M. LaValle and J.~J. Kuffner, ``Randomized kinodynamic planning,''
  \emph{International Journal of Robotics Research}, vol.~20, no.~5, pp.
  378--400, 2001.

\bibitem{karaman2010}
S.~Karaman and E.~Frazzoli, ``Incremental sampling-based algorithms for optimal
  motion planning,'' in \emph{Proceedings of Robotics: Science and Systems
  (RSS)}, 2010.

\bibitem{borrelli}
F.~Borrelli, A.~Bemporad, and M.~Morari, \emph{{Predictive Control for Linear
  and Hybrid Systems}}, Cambridge University Press, 2017.

\bibitem{dang2008}
T.~Dang, A.~Donze, O.~Maler, and N.~Shalev, ``Sensitive state-space
  exploration,'' in \emph{Proceedings of the 47th IEEE International Conference
  on Decision and Control}, 2008.

\bibitem{collision_free2}
A.~Richards, T.~Schouwenaars, J.~P. How, and E.~Feron, ``Spacecraft trajectory
  planning with avoidance constraints using mixed-integer linear programming,''
  \emph{Journal of Guidance, Control, and Dynamics}, vol.~25, no.~4, pp.
  755--764, 2002.

\bibitem{ono2011}
L.~Blackmore, M.~Ono, and B.~C. Williams, ``Chance-constrained optimal path
  planning with obstacles,'' \emph{IEEE Transactions on Robotics}, vol.~27,
  no.~6, pp. 1080--1094, 2011.

\bibitem{iss}
Z.~P. Jiang and Y.~Wang, ``Input-to-state stability for discrete-time nonlinear
  systems,'' \emph{Automatica}, vol.~37, no.~6, pp. 857--869, 2001.

\bibitem{monotone}
D.~Angeli and E.~D. Sontag, ``Monotone control systems,'' \emph{IEEE
  Transactions on Automatic Control}, vol.~48, no.~10, pp. 1684--1698, 2003.

\bibitem{baier}
C.~Baier and J.-P. Katoen, \emph{Principles of model checking}, The MIT Press,
  2008.

\bibitem{lavalle}
S.~M. LaValle, \emph{Planning algorithms}, Cambridge, UK: Cambridge University
  Press, 2006.

\bibitem{karaman2010b}
S.~Karaman and E.~Frazzoli, ``Optimal kinodynamic motion planning using
  incremental sampling-based methods,'' in \emph{Proceedings of the 49th IEEE
  Conference on Decision and Control}, 2010.

\bibitem{postoyan2015}
R.~Postoyan, J.~Biemond, W.~Heemels, and N.~van~de Wouw, ``Definitions of
  incremental stability for hybrid systems,'' in \emph{Proceedings of the 54th
  IEEE Conference on Decision and Control}, 2015, pp. 5544--5549.

\bibitem{contraction2000}
W.~Lohmiller and J.~Slotine, ``Control system design for mechanical systems
  using contraction theory,'' \emph{IEEE Transactions on Automatic Control},
  vol.~45, no.~5, pp. 884--889, 2000.

\bibitem{energyawarerouting1}
Y.~Li, Y.~Jiang, D.~Jin, L.~Su, L.~Zeng, and D.~Wu, ``Energy-efficient optimal
  opportunistic forwarding for delay-tolerant networks,'' \emph{IEEE
  Transactions on Vehicular Technology}, vol.~59, no.~9, pp. 4500--4512, 2010.

\bibitem{energyawarerouting2}
E.~Magistretti, J.~Kong, U.~Lee, M.~Gerla, P.~Bellavista, and A.~Corradi, ``A
  mobile delay-tolerant approach to long-term energy-efficient underwater
  sensor networking,'' in \emph{Proceedings of 2007 IEEE Wireless
  Communications and Networking Conference}, 2007.

\end{thebibliography}

\appendix
\noindent
\textit{(Proof for (E.2) in \rthm{second_result})}: The proof for (E.2) is also given by induction. 
Let us go back again to the initial time $k = 0$, and let
$c^* _{\ell|0}$, $\ell\in\mathbb{N}_{0:L-1}$ with $c^* _{0|0} = c_0 = 1$ be the optimal communication scheduling obtained at $k=0$. In addition, let \req{path_initial} denote the accepting run for $c^* _{\ell|0}$, $\ell\in\mathbb{N}_{0:L-1}$ and $k_1$ be the next communication time from the initial time $k=0$. 
As described in the proof of (E.1), we obtain $v_{k} \leq \gamma(s^* ({k}))$, $\forall k \in\mathbb{N}_{0: k_1}$. Moreover, since $((s^* ({k}), k), c^* _{k|0}, (s^* ({k+1}), k+1)) \in \delta_A$, $\forall k \in\mathbb{N}_{0:k_1-1}$, it follows from (D.2) and (D.3) in \rdef{transition_relation2} that 
\begin{equation}
v_{k} \leq \gamma(s^* ({k})) \leq v_{k, \max},\ \forall k \in\mathbb{N}_{0:k_1}, 
\end{equation}
which implies from \req{safety_eq} that $x_{k} \in {\cal X}$, $\forall k \in\mathbb{N}_{0: k_1}$. 
Thus, the state trajectory guarantees safety for all the time until the next communication time $k_1$. Moreover, from (D.4) in \rdef{transition_relation2}, it follows that $u_k \in {\cal U}$, $\forall k \in\mathbb{N}_{0: k_1}$ and the constraint for the control inputs are satisfied.  

Consider now the next communication time $k_1$, and let $c^* _{k_1+\ell|k_1}$, $\forall \ell \in\mathbb{N}_{0:L-k_1-1}$ be the optimal communication scheduling obtained at $k_1$ according to \req{optcom_eq}. Note that the optimal communication scheduling can be found due to the feasibility property described in (E.1). With a slight abuse of notation, let 
\begin{equation}
(s^* ({k_1}), k_1), (s^* ({1}), k_1 + 1), \ldots, (s^* ({L}), L)
\end{equation}
with $s^* ({k_1}) = s(k_1) = \mathsf{sym} (x_{k_1})$ be the accepting run for $c^* _{k_1+\ell|k_1}$, $\forall \ell \in\mathbb{N}_{0:L-k_1-1}$. 
Moreover, let $k_2 > k_1$ be the next communication time from $k_1$. 
Again, it follows from (D.2) and (D.3) in \rdef{transition_relation2} that 
\begin{equation}
v_{k} \leq \gamma(s^* ({k})) \leq v_{k, \max},\ \forall k \in\mathbb{N}_{k_1 : k_2}. 
\end{equation}
From this and (D.4) in \rdef{transition_relation2}, we obtain $x_{k} \in {\cal X}$, $u_{k} \in {\cal U}$, $\forall k \in\mathbb{N}_{k_1: k_2}$ and the safety is guaranteed until the next communication time $k_2$. 
By following the same procedure as illustrated above, we obtain $x_{k} \in {\cal X}$, $\forall k \in\mathbb{N}_{0:L}$ and $u_k \in {\cal U}$, $\forall k \in\mathbb{N}_{0:L}$. 

Let $k_N < L$ be the last communication time step, where $N$ represents the total number of communication instants. Let $c^* _{k_N+\ell|k_N}$, $\ell \in\mathbb{N}_{0:L-k_N-1}$ be the optimal communication scheduling obtained at $k_N$, and 
\begin{equation}
(s^* ({k_N}), k_N), (s^* ({k_N+1}), k_N + 1), \ldots, (s^* ({L}), L)
\end{equation}
with $s^* ({k_N}) = \mathsf{sym} (x_{k_N})$ be the accepting run for $c^* _{k_N+\ell|k_N}$, $\forall \ell \in\mathbb{N}_{0:L-k_N-1}$. 
Since $k_N$ is the last communication time step, we either have (i) $k_N = L-1$ ($c^* _{k_N|k_N}=1$), or (ii) $c^* _{k_N|k_N} = c^* _{k_N+1|k_N} = \cdots c^* _{L-1|k_N} = 0$. For case (i), it follows that 
$v_{L} \leq \gamma(s^* ({L})) \leq v_{final}$,
since $(s^* ({L}), L) \in S_{A, final}$. Thus, we obtain $x_L \in {\cal X}_F$. 
For case (ii), it follows that 
\begin{equation}
v_k \leq \gamma (s^* (k)) \leq v_{k, \max}, \ \forall k \in \mathbb{N}_{k_N : L-1} 
\end{equation}
and $v_L \leq \gamma(s^* ({L})) \leq v_{final}$. 
Thus, $x_k \in {\cal X}$, $\forall k \in \mathbb{N}_{k_N : L-1}$ and $x_L \in {\cal X}_F$. 
Therefore, the state trajectory achieves reachability and safety for both case (i) and (ii). Moreover, it follows from (D.4) in \rdef{transition_relation2} that $u_k \in {\cal U}$, $\forall k \in \mathbb{N}_{k_N : L-1}$ for both case (i) and (ii). 
Based on the above, it is shown that the state trajectory becomes valid. The proof is complete. 

\end{document}